\documentclass[a4paper,12pt]{amsart}

\usepackage[utf8]{inputenc}
\usepackage[T1]{fontenc}
\usepackage{amsmath,amssymb,amsfonts,amsthm,mathrsfs}
\usepackage{geometry}
\usepackage{graphicx}
\usepackage{float}
\usepackage[english]{babel}
\usepackage{tikz}
\usetikzlibrary{arrows}
\usetikzlibrary{decorations.markings}

\theoremstyle{plain}
\newtheorem{df}{Definition}
\newtheorem{thm}{Theorem}
\newtheorem{lem}{Lemma}
\newtheorem{prop}{Proposition}

\newtheorem{cor}{Corollary}
\newtheorem*{thm*}{Theorem}
\newcommand{\mot}{}
\newtheorem*{thmref_interne}{\mot{}}
\newenvironment{thmref}[2]{
	\renewcommand{\mot}{#1 #2}
	\begin{thmref_interne}}
	{\end{thmref_interne}
}

\newcommand\mb{\mathbb}
\newcommand\mc{\mathcal}

\newcommand\mr{\mathrm}
\newcommand\ms{\mathscr}

\renewcommand\emph[1]{\textup{\textbf{#1}}}

\title{
Irrational series I\\ Laplace transform in a neighborhood of $-\infty$
}

\author{Olivier Thom}
\date{}

\begin{document}
\maketitle

\begin{abstract}
Discrete sums of exponentials $g(w) = \sum a_{\beta} \mr{e}^{\beta w}$ with positive exponents may converge not normally in neighborhoods $H$ of $-\infty$ which do not contain half-planes.
In order to obtain a decomposition of a holomorphic function $g$ in $H$ as a sum of exponentials we study the Laplace transform in general neighborhoods of $-\infty$.

We adress questions such as continuity of Laplace and inverse Laplace transformations, continuity for the operation of taking partial sums, and resummation formulas. 
\end{abstract}

\tableofcontents
\newpage

\section{Introduction}

\subsection{Motivation}

This article originated in an effort to better understand germs of diffeomorphisms $f(z) = \lambda z + a_2 z^2 +\ldots$ in one complex variable (ideally, obtain a complete analytic classification).
The analytic classification still remains to be done in the diophantine case $\lambda = \mr{e}^{2i\pi \alpha}$ where $\alpha\in \mathbb{R}$ is irrational and well-aproximated by rational numbers.
The history of this problem can be found for example in the introduction of \cite{yoccoz_diviseurs}.

Without entering the details, let me explain the relation between diffeomorphisms and the present article.
The diffeomorphism $f$ is a priori defined in a small disk $\mb{D}$, and we can lift it to the universal cover $\mb{H}$ of the pointed disk $\mb{D}^*$.
The covering application is $w\in \mb{H} \mapsto z=\mr{e}^w\in \mb{D}^*$, and the lift of $f$ is $\tilde{f}(w) = w + 2i\pi \alpha + a_2\mr{e}^w + \ldots$.
The quotient of $\mb{H}$ by the dynamics of $\tilde{f}$ can be encoded in the uniformizing function $g$ realizing the quotient $g : \mb{H}/f \rightarrow \mb{D}^*$ (normalized to be a coordinate on $\mb{H}/f$ ; we see $g$ as a germ at $-\infty$).
This function $g$ can be shown to be bounded in a neighborhood $H\subset \mb{H}$ of $-\infty$, and in the case where $f$ is analytically linearizable, we can show that $g$ is a normally convergent series 
\begin{equation}
\label{eq_uniformizing}
g(w) = \sum_{p,q} a_{p,q} \mr{e}^{(p+q/\alpha)w}.
\end{equation}

It seems natural to suspect that even when $f$ is not linearizable, the function $g$ can still be written as a similar sum.
However, for these kind of series, normal convergence is too strong, and the convergence of $g$ must be understood in a different way.
Note for example that when $\beta\in \mathbb{R}^+$ and $h \to 0$, then 
\[
\frac{\mr{e}^{(\beta+h)w}- \mr{e}^{\beta w}}{h} \underset{h\to 0}{\longrightarrow} \frac{d}{dt}(\mr{e}^{tw})\vert_{t=\beta} = w\mr{e}^{\beta w} :
\]
the coefficients of these series can be quite big while the sum remains bounded.

As a matter of fact, this particular problem has already been studied by Écalle in \cite{ecalle_small_denominators} \footnote{A big thanks for P. Hazard for pointing out the works of Écalle and for F. Loray, G. Casale and J-M. Lion for communicating the contributions of Malgrange.}.
In this article, Écalle claims that the linearizing function $f^*$, solution of the equation $f^*\circ \tilde{f} = f^* + 2i\pi \alpha$ is of the form $f^*(w) = w + \sum c_{p,q}\mr{e}^{(p+q/\alpha)w}$.
Of course, we can switch between $f^*$ and $g$ writing $g= \mr{e}^{\alpha^{-1} f^*}$ and $f^* = \alpha\: \mr{log}(g)$.

In this article, we study the Laplace transform of a bounded function $g$ on neighborhoods $H\subset \mathbb{C}$ of $-\infty$ to be able to view $g$ as a (generalized) sum of exponentials.
In a forthcoming article, we will study more precisely the sums $\sum_{p,q\geq 0} a_{p,q} \mr{e}^{(p+q/\alpha)w}$, which we will call irrational series.
These two articles should allow for a more systematic treatment for series of exponentials, in particular we can obtain a different proof that uniformizing functions $g$ can be written as in \eqref{eq_uniformizing}.
Indeed, moving the parameter $\alpha$ of the diffeomorphism $f$ (without changing the other coefficients) allows to see $f$ as a limit of linearizable functions $f_{\alpha}$ (uniform in some neighborhood $H$), and to uniformize the quotients $\mb{H}/f_{\alpha}$ in family.
Then $g$ is the uniform limit of some uniformizing functions $g_{\alpha}$ which we know to be of the form \eqref{eq_uniformizing} : we can conclude with the continuity properties for Laplace transforms proven in this article.

\subsection{Laplace transform and hyperfunctions}

In order to decompose a function $g(w)$, bounded in a neighborhood $H\subset \mathbb{C}$ of $-\infty$, we can use Laplace transform $h(p) = \mc{L}g(p)$.
As is well-known, when $g(w) = \mr{e}^{\beta w}$ for some $\beta>0$, then $\mc{L}g(p) = \frac{1}{p-\beta}$.
We can interpret the function $\mc{L}g$ as a hyperfunction as defined by Sato in the articles \cite{sato_hyperfunctions1} and \cite{sato_hyperfunctions2}.
Hyperfunctions are some kind of distributions, more precisely they are elements in the dual of the set of analytical functions.

Let us explain roughly the idea behind the classical theory of hyperfunctions here ; the reader can learn more for example in \cite{morimoto_hyperfunctions}.
A hyperfunction supported on the compact interval $I\subset \mathbb{R}$ is an equivalence class of holomorphic functions $h\in \mc{O}(U\setminus I)$ modulo holomorphic functions on $U$ (where $U$ can be any neighborhood of $I$ in $\mathbb{C}$).
For each $h\in \mc{O}(U\setminus I)$, we can define a linear form on the space of analytic functions $\varphi$ on $I$ by the formula
\[
\frac{1}{2i\pi} \int_{\gamma} h(p) \varphi(p) dp.
\]
Here, $\gamma$ is a loop circling once around $I$ : each analytic function $\varphi$ can be extended in a neighborhood of $I$, and taking $\gamma$ close enough to $I$ this formula is well-defined and does not depend on $\gamma$.

In the case where $h(p) = \frac{1}{p-\beta}$, the linear form we obtain is the dirac $\delta_{\beta}$.
A finite sum $g(w) = \sum a_{\beta} \mr{e}^{\beta w}$ has Laplace transform $h(p) = \sum a_{\beta} \frac{1}{p-\beta}$, which corresponds to the distribution $D = \sum a_{\beta} \delta_{\beta}$.
The inverse Laplace transform consists at evaluating $D$ on the exponential function $\mr{e}_w(t)=\mr{e}^{t w}$ since $\langle D, \mr{e}_w \rangle = \sum a_{\beta} \mr{e}^{\beta w}$.

We are interested in infinite sums $\sum a_{\beta} \mr{e}^{\beta w}$, thus the integration path $\gamma$ needed to realize the inverse Laplace transform must be infinite (we will take a path "circling once around $\mathbb{R}^+$").
Because this path is not compact, we also need to modify the definition of hyperfunctions.
In fact, since we will need to change the definitions and check that everything works fine in this context, no previous knowledge about hyperfunctions will be needed in this article.

One note about Laplace transform in neighborhoods $H$ : another idea to decompose a function $g$ defined in a neighborhood $H$ is to consider an injective map $\varphi: \mb{H} \rightarrow H$, and apply the classical Borel-Laplace transform for the function $g\circ \varphi$ defined on $\mb{H}$ (here $\mb{H} = \{ w\in \mathbb{C} \:|\: \mr{Re}(w) < 0\}$).
This idea was used by Écalle in \cite{ecalle_small_denominators} (see his Proposition 3.15), or even by Ilyashenko, as presented for example in \cite{ilyashenko_centennial} (see section 3.4).
When this trick allows to answer a problem, it certainly gives a fast answer ; however Laplace transform is not very well-behaved with respect to changes of coordinates (in general, we want to decompose $g(w)$ as a sum of $\mr{e}^{\beta w}$, not as a sum of $\mr{e}^{\beta \varphi^{-1}(w)}$).
In any case, doing the Laplace transform directly in the neighborhood $H$ is certainly vantageous for clarity and to give a more systematic treatment of these questions.

\subsection{Results}

The first result of this article is Theorem \ref{thm_1}. The exact statement requires most of the definitions of section \ref{sec_definitions} ; roughly stated, it says that Laplace and inverse Laplace transforms realize a correspondence between bounded holomorphic functions on neighborhoods of $-\infty$ and hyperfunctions supported on $\mathbb{R}^+$ with exponential growth.
In this correspondence, the shape of the neighborhood corresponds to the growth of the hyperfunction. 

The next results concern functions defined in what we call logarithmic neighborhoods : they are those neighborhoods $H\subset \mathbb{C}$ of the form $ \{ x+iy \in \mathbb{C} \:|\: x\leq \rho(|y|) \}$ with $\rho(y) = w_0 - k\:\mr{log}(y)$ for $w_0,k\in \mathbb{R}$ and $y$ big enough.
In the correspondence of Theorem \ref{thm_1}, they correspond to hyperfunctions whose order grows linearly, as explained in sections \ref{sec_logarithmic_half-plane} and \ref{sec_finite_order}.

We can define partial sums of a bounded holomorphic function $g(w)$ in a neighborhood $H$ by the formula 
\[
S_{[\beta_1,\beta_2]}g(w) = \frac{1}{2i\pi} \int_{\gamma} \mc{L}g(p)\mr{e}^{pw}dp,
\]
where $\beta_1,\beta_2\in \mathbb{R}$ and $\gamma$ is a small loop circling once around $[\beta_1,\beta_2]$.
In the case where $g(w) = \sum_{\beta\in R} a_{\beta} \mr{e}^{\beta w}$, then $S_{[\beta_1,\beta_2]}g(w) = \sum_{\beta_1< \beta < \beta_2} a_{\beta} \mr{e}^{\beta w}$, if $R$ is a discrete subset of $\mathbb{R}^+$ with $\beta_1,\beta_2\notin R$.
We can define similarly remainders $S_{[\beta_1,\infty[}g(w)$.

The aplication $g \mapsto S_{[\beta_1,\beta_2]}g$ is not continuous in general : think for example of $g_t(w) = \frac{1}{2t} (\mr{e}^{(\beta_1+t)w} - \mr{e}^{(\beta_1-t)w})$.
In order to obtain continuity of partial sums, we have to suppose that the support of $\mc{L}g$ does not come close to the cutting points $\beta_1,\beta_2$ : for this purpose we introduce the set $E_0^{(I)}(H)$ of bounded holomorphic functions $g$ on $H$ such that the support of $\mc{L}g$ does not intersect $I$ (the space $E_0^{(I)}(H)$ is equipped with the supremum norm).

\begin{thmref}{Theorem}{\ref{thm_partial_sums}}
Consider a logarithmic neighborhood $H$ contained in $\{\mr{Re}(w)\leq 0\}$, two points $\beta_1, \beta_2 \in \mathbb{R}$, a number $r>0$ and $I=]\beta_1-r, \beta_1+r[\cup ]\beta_2-r,\beta_2+r[$.
The application $S_{[\beta_1,\beta_2]}: E_0^{(I)}(H) \rightarrow E_0^{(I)}(-1+H)$ is continuous.

Similarly, if $\beta_3\in \mathbb{R}^+$ and $I=]\beta_3-r,\beta_3+r[$, the application $S_{[\beta_3,+\infty[}: E_0^{(I)}(H) \rightarrow E_0^{(I)}(-1+H)$ is continuous.
\end{thmref}

In particular, for functions $g$ which are discrete sums $g(w) = \sum_{\beta\in R} a_{\beta} \mr{e}^{\beta w}$ ($R\subset \mathbb{R}^+$ discrete and fixed), the application $g \mapsto a_{\beta}$ is continuous.
The following corollary is also of interest:

\begin{thmref}{Corollary}{\ref{cor_limit_series}}
Consider subsets $R_n,R\subset \mathbb{R}^+$, enumerated by increasing sequences $(\beta_{n,k})_k$, $(\beta_k)_k$.
Consider a logarithmic half-plane $H$ and bounded functions $g_n,g\in E_0(H)$.

Suppose that $g_n \rightarrow g$ uniformly, that $\beta_{n,k} \underset{n\to\infty}{\longrightarrow} \beta_k$ for each $k$, and that the support of the hyperfunction $\mc{L}g_n$ is contained in $R_n$ for each $n$.
Then the support of $\mc{L}g$ is contained in $R$.

Moreover, for each $\beta_1,\beta_2\notin S$, the partial sums $S_{[\beta_1,\beta_2]}g_n$ are well-defined for $n$ big enough, and $S_{[\beta_1,\beta_2]}g_n \rightarrow S_{[\beta_1,\beta_2]}g$ uniformly on $-1+H$.
\end{thmref}

However, the partial sums $S_{[-1,n]}g$ do not converge to $g$ when $n$ tends to infinity, and finding summation formulas to obtain the value of a function $g(w)=\sum a_{\beta} \mr{e}^{\beta w}$ from its coefficients $a_{\beta}$ is more delicate than it seems.
We study this question in section \ref{sec_summation_formulas} : in section \ref{sec_DIPP} we introduce the notion of diagonal extraction of integration by parts for the function $g$ (DIPP for short), noted $I_1^{\Delta}(\mc{L}g,\mr{e}^{wt})$ and prove that this gives a way to compute the function $g(w)$ :

\begin{thmref}{Theorem}{\ref{thm_DIPP}}
The diagonal integration by parts $I_1^{\Delta}(\mc{L}g,\mr{e}^{wt})$ converges normally in a logarithmic neighborhood $H'$, and for every $w\in H'$, we have 
\[
\mc{L}^{-1}\mc{L}g(w) = \frac{1}{2i\pi} I_1^{\Delta}(\mc{L}g,\mr{e}^{wt}).
\]
\end{thmref}

This explains why partial sums $S_{[-1,n]}g$ do not converge to $g$ in general : looking at the DIPP what we obtain is the convergence of what we will call evanescent partial sums $\widetilde{S}_ng(w) = S_{[-1,n]}g + BT^2_n(w)$, where $BT^2_n$ are the border terms of the integration by parts on the interval $[-1,n]$.
These evanescent partial sums satisfy $g(w) - \widetilde{S}_ng(w) = O(\mr{e}^{nw})$ when $\mr{Re}(w) \rightarrow -\infty$.
They are the sum of the partial sums and the "evanescent term" $BT^2_n$ which converges formally to zero.
This result is stated in the

\begin{thmref}{Theorem}{\ref{thm_evanescent}}
The sequence of evanescent partial sums $\widetilde{S}_ng$ tends uniformly to $g$ on a logarithmic neighborhood of $-\infty$.
\end{thmref}

When $g(w) = \sum_{\beta\in R} a_{\beta} \mr{e}^{\beta w}$, the evanescent partial sum $\widetilde{S}_ng$ can be explicitely computed in function of the coefficients $a_{\beta}$, and only depend on the coefficients $a_{\beta}$ for $\beta\leq n$.

\section{Laplace transform}

\subsection{Definitions}
\label{sec_definitions}

Let us begin with some notations.
We will write the ray starting at $w_0\in \mathbb{C}$ with direction $v\in \mathbb{C}^*$ as
\[
R(w_0,v) := \{ w = w_0 + t v, t\in \mathbb{R}^+\}.
\]
We will write the left-handed open cone of vertex $w_0$ and opening $2\theta$ for $\theta\in [0,\pi/2[$ as
\[
C_-(w_0,\theta) := w_0 - \{ w\in \mathbb{C}^* \:|\: |\mr{Arg}(w)| < \theta \},
\]
and the right-handed open cone of vertex $\beta\in \mathbb{R}$ and opening $2\theta$ for $\theta\in [0,\pi/2]$ as
\[
C_+(\beta,\theta) := \beta + \{ p \in \mathbb{C}^* \:|\: |\mr{Arg}(p)| < \theta \}.
\]
We will denote by $\overline{C}_-$ or $\overline{C}_+$ the corresponding closed cones.

The open half-plane of direction $v\in \mathbb{C}^*$ will be denoted 
\[
\mathbb{H}_v := \{ p\in \mathbb{C} \:|\: \mr{Re}(p\bar{v})>0 \}.
\]

We will introduce the following function to measure the distance to $\mathbb{R}^+$
\[
N(p) := \left\{ \begin{aligned} &|p| \quad \text{if $\mr{Re}(p)\leq 0$}\\ & |\mr{Im}(p)|\quad \text{if $\mr{Re}(p)>0$} \end{aligned} \right.
\]

\begin{df}
In this paper, we will say that a subset $H\subset \mathbb{C}$ is a neighborhood of $-\infty$ if for every angle $\theta\in [0, \pi/2[$, there is a point $w_0\in H$ such that the cone $C_-(w_0,\theta)$ is contained in $H$.

For each $w_0\in H\cap \mathbb{R}$, we can introduce the biggest angle $\theta(w_0)\leq \pi/2$ such that $C_-(w_0,\theta(w_0)) \subset H$.
To simplify topological questions, we will always suppose that $H = \cup_{w_0} C_-(w_0,\theta(w_0))$ is a union of such cones, symmetric about $\mathbb{R}$.

Let us write $\ms{N}(-\infty)$ the set of all neighborhoods of $-\infty$ as above.
\end{df}

Note that when $H$ is a neighborhood of $-\infty$, the function $\theta(w)$ above satifies $\theta(w) \underset{w\to -\infty}{\longrightarrow} \pi/2$, so that the union $\cup_{w} C_-(w,\theta(w))$ is also a neighborhood of $-\infty$.

The size of the neighborhood $H$ can be measured by the function $\theta(\cdot): \mathbb{R}^- \rightarrow [0,\pi/2]$, or by the border $\partial_{}H$ of the set $H$.
When $H = \cup_{w\in ]-\infty,a]} C_-(w,\theta(w))$, the border $\partial_{}H$ can be described as 
\[
\partial_{}H = \{ x+iy\in \mathbb{C} \:|\: x = \rho(|y|) \},
\]
where $\rho: \mathbb{R}^+ \rightarrow \mathbb{R}^-$ is by definition a convex function (here, the number $a\in \mathbb{R}$ is the supremum of $H\cap \mathbb{R}$).
Of course, $\partial_{}H$ is the enveloppe of the family of curves $R(w,-\mr{e}^{\pm 2i\pi\theta(w)})$ for $w\in ]-\infty,a]$, and the correspondence between $\rho(\cdot)$ and $\theta(\cdot)$ can be found in this way.

\begin{df}
For any $H\in \ms{N}(-\infty)$, we will write $\Theta_H: w\in \mathbb{R}^-\cap H \mapsto \theta(w)\in [0,\pi/2]$ the function giving the biggest cone such that $\overline{C}_-(w,\Theta_H(w))\subset \overline{H}$.
The inverse $\Theta_H^{-1}:]0,\pi/2[ \rightarrow \mathbb{R}^-$ of the function $\Theta_H$ gives the rightmost cone $\overline{C}_-(\Theta_H^{-1}(\theta),\theta)$ of opening $\theta$ contained in $\overline{H}$.
\end{df}

Note that $\Theta_H$ and $\Theta_H^{-1}$ are inverse to one another only generically (when the function $\rho$ is smooth and strictly convex).
We can always approximate $H$ by such a neighborhood, but this won't be needed since the description given above for $\Theta_H$ and $\Theta_H^{-1}$ defines them independently.

\begin{df}
We define the family of logarithmic neighborhoods of $-\infty$ (or, logarithmic half-planes) as the set $\ms{LN}(-\infty)\subset \ms{N}(-\infty)$ of those neighborhoods for which the function $\rho(\cdot)$ satisfies $\rho(y) = w_0 -k\mr{log}(y)$ for some $k>0$ and $w_0\in \mathbb{R}$, for $y$ big enough.

The pair $(w_0,k)$ is called the type of the logarithmic neighborhood ; we denote by $\ms{LN}_{w_0,k}(-\infty)$ the set of neighborhoods of type $(w_0,k)$.
\end{df}

Of course, a neighborhood containing a logarithmic one can also be considered as a logarithmic neighborhood by restriction.
In particular, a neighborhood defined by a function $\rho$ with $\rho(y)\geq w_0-k\mr{log}(y)$ for some $w_0,k$ and any $y$ big enough can be seen as a logarithmic neighborhood.

\begin{df}
For every $H\in \ms{N}(-\infty)$, let us write $E_M(H)$ the set of functions with exponential growth of order $M$ which extend continuously to the border : a holomorphic function $g\in \mc{O}(H)$ extending continuously to the border belongs to $E_M(H)$ if there is a constant $C>0$ such that $|g(w)|\leq C \mr{e}^{M\:\mr{Re}(w)}$ for every $w\in H$.

In particular, $E_0(H)$ is the set of bounded holomorphic functions on $H$ which extends continuously to the border.
\end{df}

We will sometimes write "bounded function $g\in E_0(H)$" for short : just to be clear, all our bounded holomorphic function will be supposed to extend continuously to the border.
In fact, we only need the hypothesis of continuous extension to the border to be able to integrate along the border of $H$, which is only needed when $H$ is a straight half-plane $H=\{\mr{Re}(w) < a \}$.
In the other cases, this precaution is unnecessary.

Note that since our set $H$ is oriented towards $\mathbb{R}^-$, the functions in $E_M(H)$ for $M>0$ are actually exponentially decreasing when $\mr{Re}(w)$ tends to $-\infty$.
Since we have in mind the Laplace transform of the function $g$, and that to every $g\in E_M(H)$ we can associate the function $w\mapsto g(w)\mr{e}^{-M w}\in E_0(H)$, it will be enough to study bounded holomorphic functions in $H$.

We can topologize the vector space $E_0(H)$ with the norm $\|g\|_H = \mr{sup}_H (|g(w)|)$.

\begin{df}
An open set $V\subset \mathbb{C}$ is called an angular neighborhood of $\mathbb{R}^+$ if it contains a closed cone $\overline{C}_+(-\varepsilon,\theta)$ for some $\varepsilon,\theta>0$.

To simplify topological questions, we will always suppose that $V$ is either $\mathbb{C}$ or an open cone $C_+(-\varepsilon,\theta)$ for some $\varepsilon,\theta>0$.
To every neighborhood $V$ we will associate $V^- = V\setminus \mathbb{R}^+$ (it is still connected).

Let us write $\ms{N}(\mathbb{R}^+)$ (resp. $\ms{N}^-(\mathbb{R}^+)$) the set of all angular neighborhoods of $\mathbb{R}^+$ as above (resp. the set of all $V^-$ where $V$ is an angular neighborhood).
To simplify, we will still call those $V^-$ neighborhoods whenever this does not seem confusing.
\end{df}

\begin{df}
A holomorphic function $h\in \mc{O}(V^-)$ in a neighborhood $V^-\in \ms{N}^-(\mathbb{R}^+)$ will be said to have exponential growth along rays if for every cone $\overline{C}_+(-\varepsilon,\theta)\subset V$, there exist $C,M>0$ such that $|h(p)|\leq C\mr{e}^{M\: \mr{Re}(p)}$ for every $p\in V\setminus \overline{C}_+(-\varepsilon,\theta)$.

We will denote by $\ms{E}(V^-)$ the set of holomorphic functions with exponential growth along rays in $V^-$.

Given a function $\mu:\theta \mapsto M=\mu(\theta)$, we can introduce the subset $\ms{E}_{\mu}(V^-)\subset \ms{E}(V^-)$ for which $\mr{sup}_{\varepsilon,\theta,p} N(p)\mr{e}^{-\mu(\theta) \mr{Re}(p)}|h(p)| < \infty$, where the supremum is for every admissible $\varepsilon,\theta$ and every $p\in V\setminus \overline{C}_+(-\varepsilon,\theta)$.
\end{df}

If $h$ belongs to $\ms{E}_{\mu}(V^-)$, then we can define the norm $\|h\|_\mu$ of $h$ to be the supremum above: we then have 
\[
|h(p)| \leq \|h\|_\mu \frac{\mr{e}^{\mu(\mr{Arg}\:p)\mr{Re}(p)}}{N(p)}
\]

\begin{df}
A holomorphic function $h\in \mc{O}(V)$ in a neighborhood $V\in \ms{N}(\mathbb{R}^+)$ will be said to have exponential growth of order $M>0$ if there exists $C>0$ such that $|h(p)| \leq C\mr{e}^{M\mr{Re}(p)}$ for every $p\in V$.
We denote by $E_M(V)$ the set of those functions and $E(V) = \cup_M E_M(V)$ the set of functions with exponential growth.
\end{df}

\begin{df}
We define the set of hyperfunctions $H^1_{\mathbb{R}^+}(V,\ms{E})$ on a neighborhood $V\in \ms{N}(\mathbb{R}^+)$ as the quotient of $\ms{E}(V^-)$ by $E(V)$.
We will identify a function $h\in \ms{E}(V^-)$ with its class $[h]\in H^1_{\mathbb{R}^+}(V,\ms{E})$ whenever there is no confusion.
\end{df}

\begin{df}
We define the class $\ms{H}_\mu(V) \subset H^1_{\mathbb{R}^+}(V,\ms{E})$ as follows.
A hyperfunction $[h]$ represented by a function $h\in \ms{E}_\mu(V^-)$ belongs to $\ms{H}_\mu(V)$ if there exists a constant $C>0$ such that for every $\theta\in ]0,\frac{\pi}{2}[$, every $\varepsilon>0$ and every direction $v=-r\mr{e}^{i \alpha}$ with $|\alpha|\leq \frac{\pi}{2} - \theta$, there exists a function $f\in E(V)$ such that
\[
|h(p)-f(p)| \leq C \frac{\mr{e}^{\mu(\theta)\mr{Re}(p)-\mr{Re}(vp)}}{N(p)}
\]
for every $p\in V\setminus \overline{C}_+(-\varepsilon,\theta)$.
The smallest constant $C$ is called the norm $\| [h] \|_\mu$ of the hyperfunction $[h]$.
\end{df}

One might want to wait after the sections on Laplace and inverse Laplace transforms before trying to understand the meaning of this definition.

Logarithmic neighborhoods of $-\infty$ will correspond to the growth functions
\[
\mu_{a,k}(\theta) = a + k\:\mr{log}(\mr{cotan}(\theta)).
\]

\begin{df}
We define the class $\ms{E}_{a,k}(V^-)=\ms{E}_{\mu_{a,k}}(V^-)$ : this is the set of holomorphic functions $h$ in $V^-$ for which 
\[
|h(p)| \leq C \frac{\mr{e}^{a \mr{Re}(p)}\mr{Re}(p)^{k\mr{Re}(p)}}{|\mr{Im}(p)|^{k\mr{Re}(p)+1}}
\]
for every $p\in V^-$ with $\mr{Re}(p)>0$, where $C>0$ is a constant.

We denote by $\ms{H}_{a,k}=\ms{H}_{\mu_{a,k}}$ the corresponding hyperfunctions.
\end{df}

\subsection{Laplace transform on a neighborhood of $-\infty$}

In this section, we will consider a fixed neighborhood $H\in \ms{N}(-\infty)$ and define the Laplace transform $\mc{L}$ of a bounded holomorphic function $g\in E_0(H)$.

Consider a base point $w_0\in H$, and consider for any ray $R(w_1,v)\subset \overline{H}$ the broken path $\Gamma_{w_0,w_1,v} = [w_0,w_1]\cup R(w_1,v)$ (where $[w_0,w_1]$ can be any path joining $w_0$ and $w_1$ as long as it is included in $H$; for practical reasons we will only consider broken paths for which the segment $[w_0,w_1]$ is included in $\overline{H}$).

\begin{figure}[H]
\centering
\includegraphics[scale=0.6]{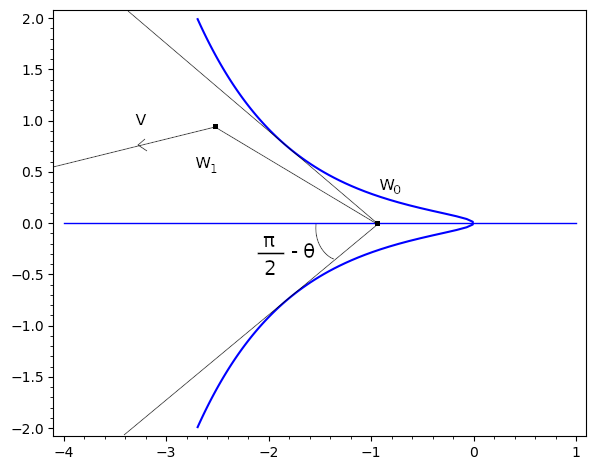}
\caption{integration path}
\end{figure}
\begin{figure}[H]
\centering
\includegraphics[scale=0.6]{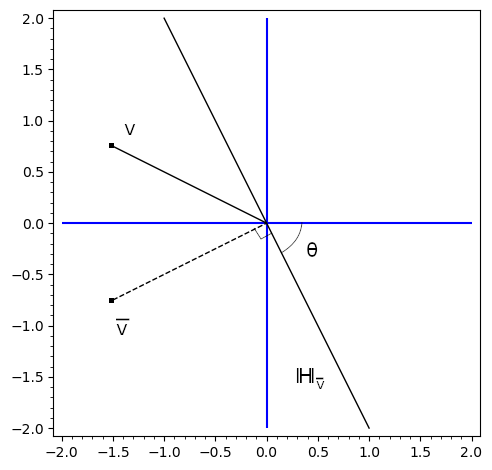}
\caption{domain of convergence}
\end{figure}

\begin{df}
For any broken path $\Gamma(w_0,w_1,v)\subset \overline{H}$, we define the Laplace transform of a bounded holomorphic function $g\in E_0(H)$ by the formula 
\[
\begin{aligned}
\mc{L}_{w_0,w_1,v}g(p) &:= \int_{\Gamma_{w_0,w_1,v}} g(w)\mr{e}^{-pw}dw\\
&= \int_{w_0}^{w_1} g(w) \mr{e}^{-pw}dw + \mr{e}^{-pw_1}\int_{t=0}^\infty g(w_1+tv) \mr{e}^{-ptv} vdt.
\end{aligned}
\]
We will also write 
\[
\mc{L}_{w_0,w_1}g(p) := \int_{w_0}^{w_1} g(w)\mr{e}^{-pw}dw.
\]
\end{df}

Since $g$ is bounded, the integral converges whenever $\mr{Re}(pv)>0$, that is for every $p$ in the open half-plane $\mb{H}_{\bar{v}}$ in the direction $\bar{v}$.

Now, whenever $p\in \mb{H}_{\bar{v}} \cap \mb{H}_{\bar{v}'}$ belongs to two such half-planes, the exponential decay of the integrand and Cauchy's theorem imply
\[
\mc{L}_{w_0,w_1,v}g(p) = \mc{L}_{w_0,w_1,v'}g(p).
\]
By the same reasons, if $R(w_1,v),R(w'_1,v)$ are two rays included in $H$ with the same direction, we have
\[
\mc{L}_{w_0,w_1,v}g(p) = \mc{L}_{w_0,w'_1,v}g(p).
\]
Thus the Laplace transform of $g$ does not depend of the direction $v$ nor on the point $w_1$, and can be defined for every $p$ in the union $\bigcup_{\mr{Re}(v)<0} \mb{H}_{\bar{v}} = \mathbb{C}\setminus \mathbb{R}^+$.
We will denote by $\mc{L}_{w_0}g\in \mc{O}(\mathbb{C}\setminus \mathbb{R}^+)$ this function.

Lastly, if $w_0,w'_0$ are two base points, we have obviously
\[
\mc{L}_{w_0}g(p) - \mc{L}_{w'_0}g(p) = \mc{L}_{w_0,w'_0}g(p) = \int_{w_0}^{w'_0} g(w)\mr{e}^{-pw}dw
\]
which is an entire function with exponential growth: $\mc{L}_{w_0,w'_0}g \in E(\mathbb{C})$.
We can then define the Laplace transform of $g$ as a hyperfunction supported by $\mathbb{R}^+$:
\[
\mc{L}g = [ \mc{L}_{w_0}g ] \in H_{\mathbb{R}^+}^1(\mathbb{C}, \ms{E}).
\]

\begin{prop}
\label{prop_Lg}
For any neighborhood $H\in \ms{N}(-\infty)$ and any bounded holomorphic function $g\in E_0(H)$, the Laplace transform $\mc{L}g$ is a well-defined hyperfunction $\mc{L}g\in H^1_{\mathbb{R}^+}(\mathbb{C},\ms{E})$.

More explicitely, for any $\theta\in ]0,\frac{\pi}{4}[$ we can consider $w_0 = \Theta^{-1}(\frac{\pi}{2}-\theta)$ so that $\overline{C}_-(w_0,\frac{\pi}{2}-\theta)\subset \overline{H}$.
For any $w_1=w_0+w_{01}\in \overline{C}_-(w_0,\frac{\pi}{2}-\theta)$ and any $p\in \mathbb{C}\setminus \overline{C}_+(0,\theta)$, we have
\[
|\mc{L}_{w_0}g(p)-\mc{L}_{w_0,w_1}g(p)| \leq 4\|g\|_H \frac{\mr{e}^{-w_0\mr{Re}(p)-\mr{Re}(w_{01}p)}}{N(p)}.
\]
In particular, $\mc{L}g$ belongs to $\ms{H}_\mu(\mathbb{C})$, where $\mu(\theta) = -w_0 = -\Theta^{-1}(\frac{\pi}{2}-\theta)$, and $\| \mc{L}g\|_\mu \leq 4\|g\|_H$.
\end{prop}

\begin{proof}
We need to estimate the growth of $\mc{L}_{w_0}g(p)$ along rays.
First of all, consider a number $p\in \mathbb{C}\setminus \overline{C}_+(0,\theta)$ and choose an angle $\alpha$ such that $|\alpha|\leq \frac{\pi}{2}-\theta$, with the associated direction $v=-\mr{e}^{i \alpha}$.
Choose a point $w_1$ such that the closed cone $\overline{C}_-(w_1,\alpha)$ be included in $H$.
We can then compute the Laplace transform along the path $\Gamma_{w_1,w_1,v}$ and obtain the following estimate:

\[
\begin{aligned}
|\mc{L}_{w_1,w_1,v}g(p)| &\leq \|g\|_H \mr{e}^{-\mr{Re}(pw_1)}\int_{t=0}^\infty \mr{e}^{-\mr{Re}(pv)t} dt\\
&\leq \|g\|_H \frac{\mr{e}^{-\mr{Re}(pw_1)}}{\mr{Re}(pv)}.
\end{aligned}
\]

Consider $\theta\in ]0,\frac{\pi}{2}[$ and $\varepsilon>0$.
Choose $w_0\in H\cap \mathbb{R}$ such that the cone $C_-(w_0,\frac{\pi}{2}-\frac{\theta}{2})$ be included in $H$.
Note first that $\mc{L}_{w_0,w_1,v}g(p) = \mc{L}_{w_0,w_1}g(p) + \mc{L}_{w_1,w_1}g(p)$ and that $\mr{Re}(pw_1) = w_0\mr{Re}(p) + \mr{Re}(pw_{01})$.

For any $p\in \mathbb{C}\setminus \overline{C}_+(0,\theta)$, we can choose the angle $\alpha$ (or the direction $v=-\mr{e}^{i \alpha}$) which maximizes $|\mr{Re}(pv)|$.
Concretely, we can decompose $\mathbb{C}\setminus \overline{C}_+(0,\theta)$ in three regions: $\mathbb{C}\setminus \overline{C}_+(0,\theta) = U_-\cup U_i\cup U_{-i}$, where $p\in U_-$ if $\mr{Arg}(p)\in ]\frac{\pi}{2}+\theta,\frac{3\pi}{2}-\theta[$; $p\in U_i$ if $\mr{Arg}(p)\in ]\theta,\frac{\pi}{2}+\theta]$; and $p\in U_{-i}$ if $\mr{Arg}(p)\in [\frac{3\pi}{2}-\theta,2\pi-\theta[$.
In the region $U_-$, we can choose $\alpha=-\mr{Arg}(p)$; in $U_i$, we will put $\alpha = \frac{3\pi}{2}- \frac{\theta}{2}$; finally, in $U_{-i}$ the natural choice is $\alpha=\frac{\pi}{2}+\frac{\theta}{2}$.

In $U_-$ we get $\mr{Re}(pv) = |p|$, in the regions $U_{\pm i}$ this gives 
\[
\mr{Re}(pv) = |p|\:\mr{sin}( |\mr{Arg}(p)- \theta/2 |)\geq |p|\mr{sin}( |\theta/2 |) \geq |p| \mr{sin}(|\theta|)/2
\]
by concavity of the sinus.
We will write this inequality as $\mr{Re}(pv)\geq |\mr{Im}(p)|/2$.

We deduce the estimate of the proposition for most of the points $p$. In the case when $\mr{Re}(p)<0$ but $p\in U_{\pm i}$, the estimate of the denominator is $\mr{Re}(pv)\geq |\mr{Im}(p)|/2$, but in this region, we have $|p|\leq 2 |\mr{Im}(p)|$ and so $\mr{Re}(pv)\geq N(p)/4$.
(This modification allows us to have a more convenient function $N(p)$ in the majoration, which does not depend on $\theta$.)

The fact that $\mc{L}_{w_0}g$ be an element of $\ms{E}(\mathbb{C}\setminus \mathbb{R}^+)$ can be obtained by any majoration of the functions $1/|p|$ and $1/|\mr{Im}(p)|$ on the complement of the cone $C_+(-\varepsilon,\theta)$.
\end{proof}

\subsection{Inverse Laplace transform}

In this section, we will consider a fixed neighborhood $V\subset \ms{N}(\mathbb{R}^+)$ and define the inverse Laplace transform $\mc{L}^{-1}$ of a function $h\in \ms{E}(V^-)$, or more precisely, of the corresponding hyperfunction $[h]\in H^1_{\mathbb{R}^+}(V,\ms{E})$.
We will not show here that $\mc{L}$ and $\mc{L}^{-1}$ are inverse to one another; the notation $\mc{L}^{-1}$ must be understood as a mere notation for now.

Consider a base point $-\varepsilon\in \mathbb{R}^-\cap V$ and a direction $v\in \mathbb{C}^*$ with $\mr{Arg}(v)\in ]0,\pi/2]$ such that the rays $R(-\varepsilon,v)$ and $R(-\varepsilon,\bar{v})$ are contained in $V$.
Introduce the path $\gamma_{\varepsilon,v} = R(-\varepsilon,v) \cup R(-\varepsilon, \bar{v})$, oriented downwards (towards the negative imaginary numbers).

\begin{figure}[H]
\centering
\includegraphics[scale=0.6]{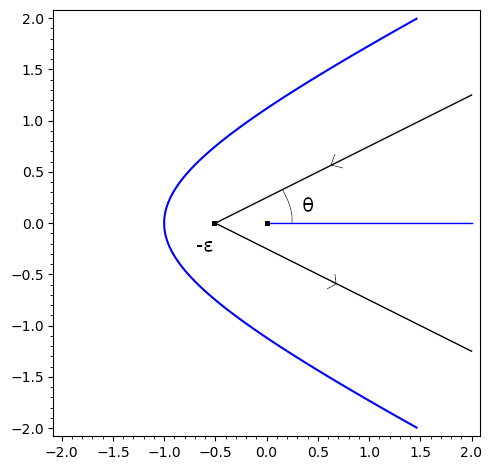}
\caption{integration path}
\end{figure}

\begin{df}
For any path $\gamma_{\varepsilon,v}$ as above, we define the inverse Laplace transform of a function $h\in \ms{E}(V^-)$ by the formula 
\[
\begin{aligned}
\mc{L}^{-1}_{\varepsilon,v}h(w) &:= \frac{1}{2i\pi} \int_{\gamma_{\varepsilon,v}}h(p)\mr{e}^{pw}dp\\
&= \frac{1}{2i\pi} \int_{t=0}^\infty h(-\varepsilon+t\bar{v}) \mr{e}^{-\varepsilon w}\mr{e}^{wt\bar{v}}\bar{v}dt - \frac{1}{2i\pi} \int_{t=0}^\infty h(-\varepsilon+t v) \mr{e}^{-\varepsilon w}\mr{e}^{wtv}vdt.
\end{aligned}
\]
\end{df}

We have an estimate $|h(p)|\leq C \mr{e}^{M\mr{Re}(p)}$ on the rays $R(-\varepsilon, v)$ and $R(-\varepsilon,\bar{v})$ by the definition of the class $\ms{E}(V^-)$.
It follows that the integral $\int_{R(-\varepsilon,v)} h(p)\mr{e}^{pw}dp$ converges when $\mr{Re}(wv)< - M\mr{Re}(v)$, that is, on the half-plane $-M - \mb{H}_{-\bar{v}}$.

The integral $\int_{R(-\varepsilon,\bar{v})} h(p)\mr{e}^{pw}dp$ converges on the symmetric half-plane $-M - \mb{H}_{-v}$, and the inverse Laplace transform is well-defined on the intersection of the two, that is, on the cone $C_-(-M,\pi/2 - \theta)$ where $\mr{e}^{i \theta} = v/|v|$.

Once again by Cauchy's theorem, the inverse Laplace transform does not depend on $\varepsilon$ nor $v$, and can be defined on the union of all the cones obtain before.
Note that the cones above can have openings arbitrarily close to $\pi/2$ by taking $v$ close to $1\in \mathbb{R}^+$, so that the inverse Laplace transform of $h$ is defined in a neighborhood of $-\infty$.

Note also that if $h$ belongs to the class $E(V)$, then we have $\mc{L}^{-1}h = 0$ by Cauchy's theorem.

We can then define the Laplace transform of a hyperfunction $[h]$ as the Laplace transform of any of its representatives :
\[
\mc{L}^{-1}[h] = \mc{L}^{-1}_{\varepsilon,v}h \in \mc{O}(H),
\]
where $H$ is a neighborhood of $-\infty$.

\begin{prop}
For any neighborhood $V\in \ms{N}(\mathbb{R}^+)$ and any function $h\in \ms{E}(V^-)$, the inverse Laplace transform of $h$ is a function with exponential growth in some neighborhood of $-\infty$.
\end{prop}

\begin{proof}
We need to estimate the integral defining the transform.
Suppose that $|h(p)|\leq C\mr{e}^{M\mr{Re}(p)}$ for all $p$ outside $C_+(-\varepsilon,\theta)$, and take the direction $v$ of angle $\theta$.
Then
\[
\begin{aligned}
|\int_{t=0}^\infty h(-\varepsilon+tv)\mr{e}^{-\varepsilon w} \mr{e}^{wtv}vdt| &\leq C\mr{e}^{-\varepsilon \mr{Re}(w)-\varepsilon M} \int_{t=0}^\infty \mr{e}^{Mt\mr{Re}(v) + t\mr{Re}(wv)} dt\\
&\leq C\frac{\mr{e}^{-\varepsilon (M+\mr{Re}(w))}}{|\mr{Re}(wv)+M\mr{Re}(v)|}
\end{aligned}
\]
when $\mr{Re}(wv)+M\mr{Re}(v) < 0$.

Thus, for $w\in C_-(-M-r,\frac{\pi}{2}-\theta)$, we have $\mr{Re}(wv)+M\mr{Re}(v) < -r\mr{Re}(v)=-r\mr{cos}(\theta)$.
For $r>0$ big enough, we will have $C \leq |r \mr{cos}(\theta)|$, and
\[
|\mc{L}^{-1}_{\varepsilon,v}h(w)| \leq \frac{2C}{2\pi} \frac{\mr{e}^{-\varepsilon(M+\mr{Re}(w))}}{r\mr{cos}(\theta)}\leq \frac{1}{\pi} \mr{e}^{-\varepsilon \mr{Re}(w)},
\]
showing that $\mc{L}^{-1}h$ is of exponential growth of order $\varepsilon$ on the neighborhood of $-\infty$ obtained as the union $\cup_{\theta} C_-(-M-r,\theta)$ (where $M$ and $r$ depend on $\theta$).
\end{proof}

\begin{prop}
\label{prop_Lh}
If moreover $[h]\in \ms{H}_\mu(V)$ for some function $\mu$, if $V\subset C_+(a,\theta_0)$ for some $\theta_0<\pi/2$, then for each cone $\overline{C}_+(-\varepsilon,\theta)\subset V$, we can integrate on the border of this cone (that is, consider the transform $\mc{L}^{-1}_{\varepsilon,v}$ for $v=\mr{e}^{i \theta}$).
If $w_0 = -\mu(\theta)-1$ and $w\in \overline{C}_-(w_0,\frac{\pi}{2}-\theta)$, we have 
\[
|\mc{L}^{-1}_{\varepsilon,v}h(w)| \leq K(V) \| [h] \|_\mu
\]
for some constant $K(V)$ which only depends on the neighborhood $V$.

When $V= \mathbb{C}$, the constant $K$ can be taken to be $K(\mathbb{C})=2$: we see that $\mc{L}^{-1}h$ is a bounded holomorphic function on the neighborhood $H$ of $-\infty$ given by $H = \cup_{\theta< \theta_0} \overline{C}_-(-\mu(\theta)-1,\frac{\pi}{2}-\theta)$, $\mc{L}^{-1}h$ extends continuously to the border and $\| \mc{L}^{-1} h \|_H \leq 2 \| [h] \|_\mu$.
\end{prop}

\begin{proof}
For every $\theta < \theta_0$, we can consider $M=\mu(\theta)$; for every $w=w_0 + w_{01}$ with $|\mr{Arg}(-w_{01})|\leq \frac{\pi}{2}-\theta$ there exists another $\tilde{h}\in [h]$ such that for every $p\in V\setminus \overline{C}_+(0,\theta)$, we have 
\begin{equation}
\label{eq_majoration_gamma-}
\begin{aligned}
|\tilde{h}(p)\mr{e}^{pw}| &\leq C \frac{\mr{e}^{M\mr{Re}(p)-\mr{Re}(w_{01}p)}}{N(p)} \mr{e}^{\mr{Re}(pw)} \\
&\leq C \frac{\mr{e}^{(M+w_0)\mr{Re}(p)}}{N(p)}. 
\end{aligned}
\end{equation}

Denote by $T_r$ the tube $T_r = \{ p\in \mathbb{C} \:|\: d(p,\mathbb{R}^+)\leq r\}$ and suppose that $T_r\subset V$ for some $r>0$.
Consider the direction $v= \mr{e}^{i \theta}$ and the path $\gamma_r = \gamma_{r,i}\cup \gamma_{r,-}\cup \gamma_{r,-i}$ where $\gamma_{r,i}$ is parametrized by $\gamma_{r,i}(t) = ir + tv$ and oriented negatively, the path $\gamma_-$ is the left half-circle of center $0$ and radius $r$, and $\gamma_{r,-i}$ is parametrized by $\gamma_{r,-i}(t) = -ir + t \bar{v}$.
Take $\theta, M$, $w=w_0+w_{01}$ and $\tilde{h}$ as above, and suppose that $w_0=-M-1$ (so that $M+w_0=-1$).
By Cauchy's theorem, the inverse Laplace transform can also be computed along the path $\gamma_r$.

The integral of $\tilde{h}(p)\mr{e}^{pw}$ along the half-circle $\gamma_{r,-}$ can obviously be estimated with Equation \eqref{eq_majoration_gamma-} by
\[
\left| \int_{\gamma_{r,-}} \tilde{h}(p)\mr{e}^{pw} dp \right| \leq C\pi\mr{e}^{r}.
\]

Along the paths $\gamma_{r,\pm i}$ we have similar estimates :
\[
\begin{aligned}
\left|\int_{\gamma_{r,i}} \tilde{h}(p)\mr{e}^{pw} dp \right| &\leq C \left|\int_{\gamma_{r,i}} \frac{\mr{e}^{-\mr{Re}(p)}}{N(p)}dp \right|\\
&\leq C \int_{t=0}^\infty \frac{\mr{e}^{-t\mr{Re}(v)}}{r}\\
&\leq \frac{C}{r\mr{cos}(\theta)}\\
&\leq \frac{C\sqrt{2}}{r}.
\end{aligned}
\]
We deduce that 
\[
|\mc{L}^{-1}h(w)| = |\mc{L}^{-1}\tilde{h}(w)| \leq \frac{C}{2\pi} (\pi\mr{e}^r + 2\sqrt{2} r^{-1}).
\]
In the case where $V=\mathbb{C}$, we can take $r=1$ to obtain 
\[
|\mc{L}^{-1}h(w)| \leq 2C.
\]
These inequalities are valid for $w \in \overline{C}_-(w_0-1, \frac{\pi}{2}-\theta)$, so when $\theta$ varies from $\theta_0$ to $0$, these sets cover the union $H = \cup_{\theta} \overline{C}_-(-\mu(\theta)-1,\frac{\pi}{2}-\theta)$ and the estimate for $\mc{L}^{-1}h$ is valid on the whole of $H$.

Taking $w_0$ such that $w_0+M = -1+\varepsilon$, the same computations show that $\mc{L}^{-1}h$ in fact extends holomorphically to a bigger open set $H'\supset \overline{H}$ ; in particular, $\mc{L}^{-1}h$ extends continuously to the border of $H$.
\end{proof}

\subsection{Properties of Laplace transforms}

The following properties of Laplace transforms are quite classical; since our context is slightly different, we state here those that we will need.
Letters $H$ and $V$ will denote respectively a neighborhood of $-\infty$ and a neighborhood of $\mathbb{R}^+$.

The first two lemmas are quite obvious.

\begin{lem}
\label{lem_exp_translation}
For every $a>0$ and every $g\in E_0(H)$, the function $\tilde{g}(w) = \mr{e}^{aw}g(w)$ is also bounded, and has Laplace transform $\mc{L}_{w_0}\tilde{g}(p) = \mc{L}_{w_0}g(p-a)$.
\end{lem}

\begin{lem}
\label{lem_translation_exp}
For every $a>0$ and every $h\in \ms{E}(V^-)$, the function $\tilde{h}(p) = h(p-a)$ is also in the class $\ms{E}(V^-)$, and its inverse Laplace transform is $\mc{L}^{-1}\tilde{h}(w) = \mr{e}^{-aw}\mc{L}^{-1}h(w)$.
\end{lem}

Derivatives and Laplace transform behave as expected (modulo a entire function with exponential growth).

\begin{lem}
For every $g\in E_0(H)$ with exponential decay, the function $\tilde{g}(w) = wg(w)$ is bounded in some neighborhood $H'\subset H$, and we have $(\mc{L}_{w_0}g)'(p) = -\mc{L}_{w_0}\tilde{g}(p)$.
\end{lem}

\begin{proof}
We just need to check that $w\mr{e}^{-aw}$ is bounded in a neighborhood $H'$ when $a>0$ : the actual computation is the classical one.
Writing $w=x+iy$, we have 
\[
\mr{log} |w\mr{e}^{-aw}|^2 = \mr{log}(x^2+y^2) - 2a x.
\]
The set $H'$ where this quantity is negative is a neighborhood of $-\infty$.
\end{proof}

\begin{lem}
For every $g\in E_0(H)$, the derivative $g'$ is bounded on $-1+H$ and satisfies $\mc{L}_{w_0} g'(p) = -g(w_0)\mr{e}^{-pw_0} + p\mc{L}_{w_0}(p)$ for any $w_0\in -1+H$.
In particular, $\mc{L} g' = p\mc{L}g$ modulo an entire function with exponential growth.
\end{lem}

\begin{proof}
The fact that $g'$ be bounded is a consequence of Cauchy's formula for the derivative.
For $\mr{Re}(p)<0$, we can define the transform using the direction $v=-1$ and integrate by parts:
\[
\begin{aligned}
\mc{L}_{w_0}g'(p) &= \int_{\Gamma_{w_0,w_0,-1}} g'(w)\mr{e}^{-pw}dw\\
&= \int_{t=0}^{+\infty} g'(w_0-t)\mr{e}^{-pw_0+pt} (-dt)\\
&= -[-g(w_0-t)\mr{e}^{-pw_0+pt}]_0^{+\infty} -\int_{t=0}^{+\infty} (-g(w_0-t))p\mr{e}^{-pw_0+pt}(-dt)\\
&= -g(w_0)\mr{e}^{-pw_0} + p\mc{L}_{w_0}g(p).
\end{aligned}
\]
We conculde by analytic continuation.
\end{proof}

If $g\in E_0(H)$ has exponential decay (i.e. $g\in \cup_{a<0}E_a(H)$), then we can define $Ig\in E_0(H)$ to be its primitive which satisfies $\underset{\mr{Re}(w)\to -\infty}{\mathrm{lim}}\:Ig(w) = 0$.

\begin{lem}
For every $g\in E_0(H)$ with exponential decay, its primitive $Ig$ is bounded and satisfies $\mc{L}_{w_0} Ig (p) = \frac{\mc{L}_{w_0}g(p)}{p}$.
\end{lem}

\begin{proof}
We can write
\[
Ig(w) = - \int_{\Gamma_{w,w_0,-1}} g(z)dz
\]
for any point $w_0\in H$ (we can take $w_0\in \mathbb{R}^-$ to simplify), and for $\mr{Re}(p)<0$, we can compute $\mc{L}Ig(p)$ along the broken path $\Gamma_{w_0,w_0,-1}$.
We obtain
\[
\begin{aligned}
\mc{L}_{w_0} Ig(p) &= -\int_{w\in \Gamma_{w_0,w_0,-1}}\int_{z\in \Gamma_{w,w_0,-1}} g(z)\mr{e}^{-pw}dzdw\\
&= -\int_{z\in \Gamma_{w_0,w_0,-1}}\int_{w\in \Gamma{z,w_0,-1}} g(z)\mr{e}^{-pw}dwdz\\
&= \int_{z\in \Gamma_{w_0,w_0,-1}} g(z) \frac{\mr{e}^{-pz}}{p}dz
\end{aligned}
\]
The result follows for $\mr{Re}(p)<0$ ; we conclude by analytic continuation.
\end{proof}

\begin{lem}
\label{lem_ILg}
Suppose $H\subset \{\mr{Re}(w)\leq -1\}$.
For every $g\in E_0(H)$, the transform $\mc{L}g(p)$ has exponential decay when $\mr{Re}(p) \to -\infty$, and $I\mc{L}_{w_0}g = - \mc{L}_{w_0} \left( \frac{g(w)}{w} \right)$.
\end{lem}

\begin{proof}
The function $\tilde{g}(w)=-g(w)/w$ is itself bounded ; $\mr{e}^{-w}\tilde{g}(w)$ has exponential decay so that $\mc{L}_{w_0}(w\mr{e}^{-w}\tilde{g}(w)) = -(\mc{L}_{w_0}(\mr{e}^{-w}\tilde{g}(w)))' = - (\mc{L}_{w_0}\tilde{g})'(p-1)$.
We deduce $\mc{L}_{w_0}g(p-1) = -(\mc{L}_{w_0}\tilde{g})'(p-1)$ for every $p$ ; since $\mc{L}_{w_0} \tilde{g}(p)$ and $-I\mc{L}_{w_0}g(p)$ have the same derivative and both tend to 0 when $\mr{Re}(p)$ tends to $-\infty$, they are equal.
\end{proof}

The Laplace transform of an inverse function is also quite classical, but let us remind it for the sake of completeness.

\begin{lem}
\label{lem_laplace_inverse_1}
The Laplace transform of the function $i_z: w\in H \mapsto \frac{1}{z-w}$ where $z\notin H$ is equal to $\mc{L}i_z(p) = \mr{e}^{-zp}\mr{log}(p)$ modulo an entire function of exponential growth.
\end{lem}

\begin{proof}
Fix a point $w_0\in H$ and an integration path $\Gamma = \Gamma_{w_0,w_0,v}$ inside $H$.
Denote by $h(p) = \mc{L}_{w_0}i_z(p)$; the derivative of $h$ is
\[
h'(p) = \int_{\Gamma} \frac{-w\mr{e}^{-pw}}{z-w}dw,
\]
so that 
\[
zh(p)+h'(p) = \int_{\Gamma} \mr{e}^{-pw}dw = \frac{\mr{e}^{-pw_0}}{p}.
\]
It follows that 
\[
\frac{d}{dp}(\mr{e}^{zp}h(p)) = \frac{\mr{e}^{p(z-w_0)}}{p} = \frac{1}{p} + \frac{\mr{e}^{p(z-w_0)}-1}{p} = \frac{1}{p} + a(p).
\]
Note that $a(p)$ is a holomorphic function with exponential growth on $\mathbb{C}$; as such it has a primitive $A(p)$ which is entire with exponential growth.
Including the constants of integration in the function $A$, we conclude that
\[
h(p) = \mr{e}^{-zp}\mr{log}(p)+\mr{e}^{-zp}A(p).
\]
Remember that the logarithm has to be seen as a holomorphic function on $\mathbb{C}\setminus \mathbb{R}^+$ (whatever the determination).
\end{proof}

\begin{lem}
\label{lem_laplace_inverse_2}
Consider the function $l_z: p\in \mathbb{C}\setminus \mathbb{R}^+ \mapsto \mr{e}^{-zp}\mr{log}(p)$.
Its inverse Laplace transform is $\mc{L}^{-1}l_z(w) = \frac{1}{z-w}$.
\end{lem}

\begin{proof}
Fix $\varepsilon\in ]0,1[$ and $\theta>0$ to define the transform $\mc{L}^{-1}_{\varepsilon,\theta}$.
Then for any $w\in C_-(\mr{Re}(z)-1,\frac{\pi}{2}-\theta)$, we see that the function $l_z(p)\mr{e}^{pw}$ has exponential decay when $\mr{Re}(p) \rightarrow +\infty$, and the transform $\mc{L}^{-1}_{\varepsilon,\theta}l_z(w)$ is well-defined.
It follows that $\mc{L}^{-1}_{\varepsilon,\theta'}l_z(w) = \mc{L}^{-1}_{\varepsilon,\theta}l_z(w)$ for any $\theta'<\theta$.
Letting $\theta'$ tend to zero we find 
\[
\mr{lim}_{\theta'\to 0} \mc{L}^{-1}_{\varepsilon,\theta'}l_z(w) = \int_{p\in \mathbb{R}^+} \mr{e}^{p(w-z)} dp = \frac{1}{z-w}.
\]
We conclude by analytic continuation that $\mc{L}^{-1}l_z(w) = \frac{1}{z-w}$ on the domain of definition of $\mc{L}^{-1}l_z$.
\end{proof}

\subsection{Duality between $\mc{L}$ and $\mc{L}^{-1}$}

We will prove that $\mc{L}$ and $\mc{L}^{-1}$ are inverse to one another in an indirect manner: we first prove an avatar of Cauchy's formula for functions with exponential decay in $H$, and then conclude using the fact that duality works for inverse functions $i_z(w) = 1/(z-w)$ and the continuity of the transforms.

In the following lemma, we will use a curve $\ms{C}\subset \mathbb{C}$ of the form $\ms{C} = \{ x+iy\in \mathbb{C} \:|\: x = \rho(|y|) \}$.
We will choose a function $\rho: \mathbb{R}^+ \rightarrow \mathbb{R}$ such that $\ms{C}\subset H$, $H_{\rho} := \{ x+iy \:|\: x\leq \rho(|y|)\}$ belong to $\ms{N}(-\infty)$ (that is, $\rho$ is sub-linear and convex), and $\rho(y)\leq -\mr{log}(1+y)$.
(For logarithmic neighborhoods, we can take $\rho(y) = C - k\:\mr{log}(1+y)$ with $k\geq 1$ and $C\leq 0$.)

In fact, the curve $\ms{C}$ can be thought of as the border of $H$, but when $H$ is too big (a half-plane for example), the integral below might diverge and we have to choose a sharper curve $\ms{C}$.

\begin{lem}
\label{lem_integral_border}
Suppose $g$ has exponential decay: $|g(w)|\leq C \mr{e}^{a\mr{Re}(w)}$ for some constants $C>0$ and $a\geq 2$.
Consider a function $\rho$ as above and the corresponding curve $\ms{C}$, border of the neighborhood $H_{\rho}\subset H$.
Then for every $w\in H_{\rho}$, we have
\[
2i\pi g(w) = \int_{\ms{C}} \frac{g(z)}{z-w}dz.
\]
The integral converges uniformly on every neighborhood $H_{\rho-\varepsilon}\subset H_{\rho}$.
\end{lem}

\begin{proof}
We can suppose $H=H_{\rho}$ in this proof to simplify.
We can obtain the formula by integrating on the border $\ms{C}_R$ of the intersection $H\cap D_R$ between $H$ and a disk $D_R$, and letting $R$ tend to infinity.
Decomposing $\ms{C}_R = (\ms{C}\cap D_R) \cup S_R$, we must show that the integral along the arc $S_R$ tends to zero, and that the integral along $\ms{C}$ converges.

First, the integral along $\ms{C}$ can be written 
\[
\int_{\ms{C}} \frac{g(z)}{z-w}dz = \int_{y=-\infty}^\infty \frac{g(z(y))}{z(y)-w}z'(y)d y,
\]
where 
\[
z(y) = \rho(|y|) + iy.
\]
We see that
\[
|g(z(y))| \leq C \mr{e}^{a \rho(|y|)}\leq C \frac{1}{(1+|y|)^a}.
\]
The function $z\mapsto \frac{1}{z-w}$ is bounded on $H$; the function $y \mapsto \rho'(y)$ tends to zero when $y$ tends to infinity since $\rho$ is sub-linear, so $z'(y)$ is also bounded.
It follows that the integral indeed converges when $a\geq 2$.

Next, note that $S_R$ is included in a half-plane $\{\mr{Re}(z) \leq \mr{Re}(z(y))\}$ where $z(y)$ is the intersection point between the circle $C_R$ of radius $R$ and $\ms{C}$.
The coordinates $z(y) = x +iy$ of this point satisfy the equations 
\[
\left\{\begin{aligned}
x^2 + y^2 &= R^2\\
x &= \rho(y) \leq -\mr{log}(1+|y|).
\end{aligned} \right.
\]
Since $\rho$ is sub-linear we have $y \sim R$ and so $\mr{Re}(z(y)) \leq c-\mr{log}(1+R)$ for some constant $c$.
The path $S_R$ has length at most $\pi R$ and the function under the integral is once again bounded by 
\[
\frac{|g(z)|}{|z-w|} \leq C \frac{c^a}{(1+R)^a},
\]
it follows that whenever $a\geq 2$ the integral along $S_R$ tends to zero.

This proves the simple convergence of the integral ; to obtain uniform convergence, we only need a uniform bound on the function $z \mapsto \frac{1}{z-w}$, which can easily be obtained if $w\in H_{\rho-\varepsilon}$.
\end{proof}

\begin{lem}
\label{lem_laplace_correspondence_decay}
Suppose $g\in E_0(H)$ has exponential decay in $H$ : $|g(w)|\leq C\mr{e}^{a\mr{Re}(w)}$ for some $C>0$ and $a\geq 2$.
Consider a function $\rho$ as above and the corresponding neighborhood $H_{\rho}$.

The inverse Laplace transform $\mc{L}^{-1}\mc{L}g$ is defined and bounded in $H_{\rho-1}$, and for every $w\in H_{\rho-1}$, we have 
\[
\mc{L}^{-1}\mc{L}g(w) = g(w).
\]
\end{lem}

\begin{proof}
We already know by Propositions \ref{prop_Lg} and \ref{prop_Lh} that $\mc{L}^{-1}\mc{L}g$ is well-defined and bounded in $H_{\rho-1}$ ; we just need to show the formula.

By the integral formula of Lemma \ref{lem_integral_border}, we can see $2i\pi g(w)$ as a limit of Riemann sums 
\[
2i\pi g_n(w) = \frac{1}{n}\sum_{k=1}^n \frac{g(z_{n,k})}{z_{n,k}-w},
\]
and the convergence is uniform in $H_{\rho-\varepsilon}$.
We have $\mc{L}^{-1}\mc{L} g_n(w) = g_n(w)$ for every $w\in H_{\rho}$ from the Laplace transforms of inverse functions (Lemmas \ref{lem_laplace_inverse_1} and \ref{lem_laplace_inverse_2}), and the linearity of $\mc{L}$ and $\mc{L}^{-1}$.
The transform $g \mapsto \mc{L}^{-1}\mc{L}g$ is a continuous operator from $E_0(H_{\rho-\varepsilon})$ to $E_0(H_{\rho-1-\varepsilon})$ from the estimates in Propositions \ref{prop_Lg} and \ref{prop_Lh}.
Since $g_n \rightarrow g$ when $n\to \infty$ in the space $E_0(H_{\rho-\varepsilon})$, we obtain by continuity $\mc{L}^{-1}\mc{L}g(w) = g(w)$ for $w\in H_{\rho-1-\varepsilon}$.
This is true for every $\varepsilon>0$, thus for every $w\in H_{\rho-1}$.
\end{proof}

The Laplace transform and the inverse Laplace transform realize a correspondence between bounded holomorphic functions on neighborhoods of $-\infty$ and hyperfunctions of exponential growth along rays in a neighborhood of $\mathbb{R}^+$.

\begin{thm}
\label{thm_1}
If $g$ is bounded in some neighborhood $H\in \ms{N}(-\infty)$, extends continuously to the border, and $H$ is described by the function $\Theta_H$, then $\mc{L}g$ is well-defined on $\mathbb{C}\setminus \mathbb{R}^+$ and the hyperfunction $[\mc{L}g]$ belongs to the class $\ms{H}_\mu$ for $\mu(\theta) = - \Theta^{-1}(\frac{\pi}{2}-\theta)$.
We have $\mc{L}^{-1}\mc{L}g = g$ in the neighborhood $-1+H$.

If $[h]$ is a hyperfunction in the class $\ms{H}_\mu$ for some function $\mu$, then $\mc{L}^{-1}[h]$ is a bounded holomorphic function in the neighborhood $H_{\Theta}$ described by the function $\Theta$ given by $\Theta^{-1}(\theta) = -\mu(\frac{\pi}{2} - \theta)-1$.
We have $\mc{L}\mc{L}^{-1}[h] = [h]$ as hyperfunctions.
\end{thm}

\begin{proof}
We only need to reduce the general case to Lemma \ref{lem_laplace_correspondence_decay}.
Given a bounded function $g$ in some neighborhood $H$, we can consider $\tilde{g}(w) = \mr{e}^{-2w}g(w)$ which has exponential decay of order 2, and some function $\rho$ as needed to apply Lemma \ref{lem_laplace_correspondence_decay}.
We obtain $\mc{L}^{-1}\mc{L} \tilde{g} = \tilde{g}$ in $H_{\rho-1}$.

By Lemma \ref{lem_exp_translation}, we have $\mc{L} \tilde{g}(p) = \mc{L}g(p-2)$; by Lemma \ref{lem_translation_exp}, we have $\mc{L}^{-1} \mc{L} \tilde{g}(w) = \mr{e}^{-2w}\mc{L}^{-1}\mc{L}g(w)$.
It follows that $\mc{L}^{-1}\mc{L}g(w) = g(w)$ for every $w\in H_{\rho-1}$.
However, $\mc{L}^{-1}\mc{L}g$ is defined in $-1+H$ by Propositions \ref{prop_Lg} and \ref{prop_Lh}, and we supposed from the beginning that all our neighborhoods of $-\infty$ were connected.
By analytic continuation, we have $\mc{L}^{-1}\mc{L}g = g$ on $-1+H$.

Now, if $[h]\in \ms{H}_\mu$ is a hyperfunction, we know by Proposition \ref{prop_Lh} that $\mc{L}^{-1}h$ is a bounded holomorphic function in the neighborhood $H_{\Theta}$.
It follows from the first part that $\mc{L}^{-1}\mc{L}\mc{L}^{-1} h = \mc{L}^{-1}h$ in some neighborhood of $-\infty$.
This means that $h_0 := \mc{L}\mc{L}^{-1} h - h$ is in the kernel of $\mc{L}^{-1}$.

Remember that $\mc{L}^{-1}h_0$ were defined as a difference 
\[
2i\pi \mc{L}^{-1}h_0(w) = \int_{R(-\varepsilon,\bar{v})}h_0(p)\mr{e}^{pw}dp - \int_{R(-\varepsilon,v)}h_0(p)\mr{e}^{pw}dp.
\]
Saying that $\mc{L}^{-1}h_0 = 0$ means that both integrals are equals.
These integrals are in fact Laplace transforms $\mc{L}_{-\varepsilon}h_0(-w)$ based at the point $-\varepsilon$, and changing the path of integration as usual, we see that this defines a well-defined holomorphic function $\mc{L}_{-\varepsilon}h_0$ on the complement $\mathbb{C}\setminus K$ of a compact set $K$.
By the classical Fourier-Laplace correspondence \cite[Theorem 2.5.2]{morimoto_hyperfunctions}, this means that $h_0$ is a holomorphic function with exponential growth.
In particular, $h_0\in E(V)$ and $[h_0]=0$ as an element of $H^1_{\mathbb{R}^+}(V,\ms{E})$.
\end{proof}

\subsection{Example : logarithmic half-planes}
\label{sec_logarithmic_half-plane}

In the case of logarithmic neighborhoods of $-\infty$, we have a set $H = \{ x+iy \in \mathbb{C} \:|\: x\leq \rho(|y|) \}$ where $\rho(y) = w_0 - k\mr{log}(y)$ for $y$ big enough.
We can compute explicitely the growth function $\mu(\cdot)$.

Indeed, we can denote by $\ms{C}$ the border of $H$, by $R=R(w,-\mr{e}^{i\theta})$ the ray starting from $w\in H$ which is tangent to $\ms{C}$, and by $x+iy$ the intersection point between $\ms{C}$ and $R$.
We have by definition $x=\rho(y)$ and 
\[
\rho'(y) = \frac{x-w}{y} = -\mr{cotan}(\theta).
\]
In the case where $\rho(y) = w_0-k\mr{log}(y)$, we obtain 
\[
\begin{aligned}
&\frac{-k}{y} = \frac{w_0-w - k\mr{log}(y)}{y}\\
&\quad \Leftrightarrow w = w_0 +k - k\mr{log}(y),
\end{aligned}
\]
and 
\[
\mr{tan}(\theta) = \frac{y}{k}.
\]
In particular 
\[
\Theta^{-1}(\theta) = w_0+k-k\mr{log}(y) = (w_0+k-k\mr{log}(k)) - k\mr{log}(\mr{tan}\:\theta)
\]
The function $\mu$ is then given by 
\[
\mu(\theta) = -\Theta^{-1}(\frac{\pi}{2}-\theta) = -w_0 -k+ k\mr{log}(k) + k\mr{log}(\mr{cotan}(\theta)).
\]

Given a point $p=p_1+ip_2$ (with $p_1,p_2>0$) and $\theta = \mr{Arctan}(p_2/p_1)$, we get the following estimate for a function $h\in \ms{E}_\mu(V^-)$ :
\[
\begin{aligned}
|h(p)| &\leq \|h\|_\mu \frac{\mr{e}^{\mu(\theta)p_1}}{p_2}\\
&\leq \|h\|_\mu \frac{\mr{e}^{k'p_1}}{p_2} \left( \frac{p_1}{p_2} \right)^{kp_1}\\
&\leq \|h\|_\mu \frac{\mr{e}^{k'p_1}p_1^{kp_1}}{p_2^{kp_1+1}}.
\end{aligned}
\]

To sum up :

\begin{prop}
The class of logarithmic neighborhood $\ms{LN}(-\infty)$ correspond via Laplace transform to hyperfunctions in one of the classes $\ms{H}_{a,k}$.
\end{prop}

An important feature of these hyperfunctions is that they grow polynomially when tending to a finite real point.

\subsection{Example : straight half-planes}
\label{sec_straight_half-plane}

This case is classical, let us just check that everything works as expected.

If $H = \{ x+iy \:|\: x<a \}$, the function $ \Theta_H^{-1}$ is constant $\Theta_H^{-1}(\theta) = a$ and we obtain $\mu(\theta) = a$.
Note that we are in fact in the degenerate case of a heighborhood $H$ whose border-defining function $\rho$ is not strictly convex : both $\Theta_H$ and $\Theta_H^{-1}$ are constant.
In fact, we can fix $w_0\in \overline{H}$ and compute the Laplace transform $\mc{L}g(p)$ of $g$ on vertical lines $w_0 \pm i \mathbb{R}^+$ when $\mr{Re}(p)\geq 0$ (or on $w_0 + \mathbb{R}^-$ when $\mr{Re}(p)<0$) : we obtain immediately 
\[
|\mc{L}_{w_0}g(p)| \leq \|g\|_H \frac{\mr{e}^{-w_0 \mr{Re}(p)}}{N(p)}.
\]

We get the optimal bound integrating along the border, that is, for $w_0 = a$.
For example, bounded holomorphic functions on $\{ \mr{Re}(w) < 0\}$ extending continuously to the border give hyperfunctions $h$ with $|h(p)|\leq C / N(p)$.

\section{Continuity properties for the Laplace transform}

Recall that the support $\mr{supp}([h])\subset \mathbb{R}^+$ of a hyperfunction $[h]$ (represented by $h$) is defined by the fact that $\beta$ is not in $\mr{supp}([h])$ if $h$ can be extended to a holomorphic function in a neighborhood of $\beta$.

We would like to say that the support of a hyperfunction depends continuously on the hyperfunction, but the reality is more tricky.
For example, we can use Runge's approximation to obtain a sequence $(h_n)$ of functions on $\mathbb{C}\setminus\{\beta_0\}$ with poles only at $\beta_0\in \mathbb{R}^+$ to approximate uniformly on every compact subset of $\mathbb{C}\setminus \mathbb{R}^+$ the function $h(p) = \frac{1}{p-\beta_1}$ with $\beta_1 > \beta_0$.

However, in the case of hyperfunctions in the class $\ms{H}_{a,k}$ (corresponding to a logarithmic half-plane), the order of the hyperfunction at a finite point must be bounded ; this will be sufficient to prove that this kind of behaviour does not happen.

\subsection{Hyperfunctions of locally finite order}
\label{sec_finite_order}

In this section, $V\subset \mathbb{C}$ is an angular neighborhood of $\mathbb{R}^+$.

\begin{df}
Let $[h] \in H^1_{\mathbb{R}^+}(V,\mc{E})$ be a hyperfunction; we say that $[h]$ has order at most $N\in \mathbb{N}$ around a point $\beta\in \mathbb{R}^+$ if there exists a neighborhood $U$ of $\beta$ in $\mathbb{C}$ and a representant $h\in \mc{O}(V\setminus \mathbb{R}^+)$ of $[h]$ which satisfies the estimate 
\[
|h(p)| \leq C \frac{1}{|\mr{Im}(p)|^{N+1}}
\]
for some constant $C>0$ and every $p\in U\setminus \mathbb{R}^+$.

We say that the hyperfunction $[h]$ is of locally finite order if it is of finite order around any point $\beta\in \mathbb{R}^+$.
\end{df}

\begin{lem}[Integration trick]
\label{lem_integration_trick}
Consider a point $\beta\in \mathbb{R}^+$, a neighborhood $U\subset V$ of $\beta$, a convex open set $D\subset V$ containing $[0,\beta]\cup U$ and a neighborhood $W\subset D$ of $\mathbb{R}^+\cap D$.
For any integer $N\geq 0$ and any compact subset $K\subset U$, there is a constant $b_{K,N}$ satisfying the following.
For any function $h\in \mc{O}((V\setminus \mathbb{R}^+)\cup U)$ satisfying the estimates 
\[
\begin{aligned}
&|h(p)| \leq C_W\quad \forall p\in D\setminus W,\\
&|h(p)| \leq C_U \frac{1}{|\mr{Im}(p)|^N}\quad \forall p\in U\setminus \mathbb{R}^+,
\end{aligned}
\]
we also have
\[
|h(p)|\leq b_{K,N}\mr{max}(C_W,C_U)\quad\forall p\in K.
\]
\end{lem}

See the proof for an explicit majoration.

\begin{figure}[H]
\centering
\includegraphics[scale=0.7]{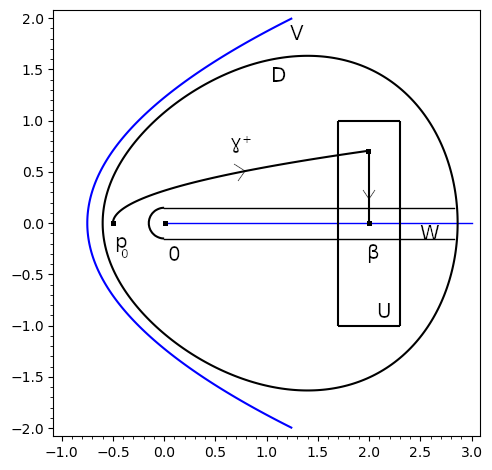}
\caption{The setup for the integration trick}
\end{figure}

\begin{proof}
Fix a base point $p_0\in D\setminus W$, and consider the primitive $H(p)$ of $h$ given by $H(p_0)=0$ (suppose for simplicity that $D\setminus W$ is simply connected).
The set $(D\setminus W) \cup U$ is no more simply connected, and there are two choices to extend $H$ to $U$: either we extend it from above (denote by $H^+$ this choice), or from below (denote that extension $H^-$).

To obtain an estimate of $H^+$ on $U\cap \{\mr{Im}(p)>0\}$, we can fix a path $\gamma^+$ starting at $p_0$ and crossing $(D\setminus W) \cap U$.
To arrive at some $p\in U\cap \{\mr{Im}(p)>0\}$, go from $p_0$ to the point $p_1$ on $\gamma^+$ vertically above $p$, and then integrate along a vertical line from $p_1$ to $p$.
We can estimate $|\int_{p_0}^{p_1} h(x)dx|\leq C_WL$ where $L$ is the length of $\gamma^+$, and
\[
|\int_{p}^{p_1} h(x)dx| \leq \int_{t=\mr{Im}(p)}^{\mr{Im}(p_1)} \frac{C_U}{t^N}dt = \frac{C_U}{N-1} \left( \frac{1}{\mr{Im}(p_1)^{N-1}} - \frac{1}{\mr{Im}(p)^{N-1}} \right).
\]
We see that the function $H$ satisfies the same kind of uniform estimates as $h$, but of order $N-1$.
By induction, we obtain similar estimates for the iterated primitives $H^{(k)}$ of $h$ (where $H^{(0)}=h$, and $H^{(k+1)}$ is a primitive of $H^{(k)}$).
In particular, $|H^{(N)}(p)|\leq C_1+C_2|\mr{log}(\mr{Im}(p))|$ for $p\in U\setminus \mathbb{R}^+$, and $|H^{(N+1)}|$ is bounded in $U\setminus \mathbb{R}^+$.

Consider Taylor's polynomial $T_{p_1}(p)$ for the function $H^{(N+1)}$ at the point $p_1$:
\[
T_{p_1}(p) = H^{(N+1)}(p_1) + (p-p_1)H^{(N)}(p_1) + \ldots + \frac{(p-p_1)^N}{N!}H^{(1)}(p_1).
\]
Taylor's formula with integral remainder then gives the following estimate for a point $p$ with $p=\mr{Re}(p_1)+ti$ and $t>0$:
\[
\begin{aligned}
|H^{(N+1)}(p) - T_{p_1}(p)| &\leq \int_{t}^{\mr{Im}(p_1)} |h(\mr{Re}(p_1)+si)|\frac{|t-s|^{N}}{N!}ds\\
&\leq \frac{C_U}{N!}\int_{t}^{\mr{Im}(p_1)} \frac{(s-t)^{N}}{s^N}ds\\
&\leq \frac{C_U}{N!}\int_t^{\mr{Im}(p_1)} 1.ds\\
&\leq \frac{C_U\mr{Im}(p_1) }{N!}
\end{aligned}
\]

We can estimate $|H^{(k)}(p_1)|$ and then $|T_{p_1}(p)|$ in function of the length $L$ of $\gamma^+$:
\[
|H^{(k)}(p_1)| \leq \frac{C_WL^k}{k!},
\]
\[
\begin{aligned}
|T_{p_1}(p)| &\leq \sum_{k=1}^{N+1} \frac{|p-p_1|^{N+1-k}}{(N+1-k)!}\frac{C_WL^k}{k!}\\
&\leq \frac{C_W}{(N+1)!} \left( L + |p-p_1| \right)^{N+1}.
\end{aligned}
\]

Now, recall that $h$ was in fact defined on the whole of $U$, so that its primitives $H^+$ and $H^-$ above and below $\mathbb{R}^+$ in fact differ from a constant.
More generally, $H^{(N+1)+}$ and $H^{(N+1)-}$ are in fact defined on the whole of $U$ and differ from a polynomial of degree $N$:
\[
H^{(N+1)+}(p) - H^{(N+1)-}(p) = P(p).
\]

In restriction to the real interval $I := U\cap \mathbb{R}$, we can use the latter estimates: 
\[
\begin{aligned}
|P(p)| &\leq |H^{(N+1)+}(p)| + |H^{(N+1)-}(p)| \\
&\leq \frac{C_U(\mr{Im}(p_1^+)-\mr{Im}(p_1^-))}{N!} + \frac{C_W}{(N+1)!} \left( (L+|p-p_1^+|)^{N+1}+(L+|p-p_1^-|)^{N+1} \right)\\
&\leq \frac{2C_UB}{N!}+ \frac{2C_W}{(N+1)!}(L+B)^{N+1},
\end{aligned}
\]
if $U$ is a rectangle with width $A$ and height $2B$.
Using Lemma \ref{lem_interpolation_rectangle} stated thereafter allows us to get an estimate on the whole of $U$:
\[
\|P\|_U \leq \frac{2C_UBk^N }{N!} + \frac{2C_Wk^N(L+B)^{N+1}}{(N+1)!},
\]
where $k=2e\sqrt{1+ (B/A)^2}$.
We get an estimate for $H^{(N+1)+}(p)$ in the region $\mr{Im}(p)<0$ using the equality $H^{(N+1)+} = H^{(N+1)-}+P$:
\[
|H^{(N+1)+}(p)| \leq \frac{C_UB(1+2k^N)}{N!}+ \frac{C_W(1+2k^N)}{(N+1)!}(L+B)^{N+1}.
\]

To obtain an estimate of $|h|$ in $U$, we use Cauchy's formula for the $(n+1)$-th derivative of $H^{(N+1)+}$: consider a point $p\in U$ at distance $d(p,\partial_{}U) > \ell$ from the border of $U$, and a circle $S$ of radius $\ell$ centered at $p$.
Then 
\[
\begin{aligned}
\frac{|h(p)|}{(N+1)!} &\leq \frac{1}{2\pi}\int_{z\in S} \frac{|H^{(N+1)+}(z)|}{|z-p|^{N+2}} |dz|\\
&\leq \frac{C_UB(1+2k^N)}{\ell^{N+2}N!}+\frac{C_W(1+2k^N)(L+B)^{N+1}}{\ell^{N+2}(N+1)!}.
\end{aligned}
\]
\end{proof}

\begin{lem}
\label{lem_interpolation_rectangle}
Consider a rectangular domain $U$ symmetric about $\mathbb{R}$.
Suppose that $I := U\cap \mathbb{R}$ has length $A$, $U$ has height $2B$ and put $k := 2e \sqrt{1+(B/A)^2}$.
Then for any polynomial $P$ of degree at most $N$, we have
\[
\mr{sup}_{x\in U} |P(x)| \leq k^N \mr{sup}_{x\in I} |P(x)|.
\]
\end{lem}

\begin{proof}
Write $I=[a,b]$ and introduce for $k=0,\ldots,N$ the points 
\[
x_k = a + k \frac{b-a}{N}.
\]
We will use interpolation at the points $x_k$ to write a polynomial $P$: introduce 
\[
Q_n(x) = \frac{\prod_{k\neq n}(x-x_k)}{\prod_{k\neq n}(x_n-x_k)},
\]
so that for every polynomial $P$ of degree at most $N$, we have
\[
P(x) = \sum_{n=0}^{N} P(x_n) Q_n(x).
\]
To estimate the norm of $P$ on $U$ we only need to estimate the norm of $Q_n$.
Note that for $k\neq n$ we have
\[
|x_n-x_k| \leq |n-k| \frac{A}{N},
\]
and that for every $x\in U$, the distance between $x$ and some $x_n$ is at most the diagonal of a half-rectangle: $|x-x_n| \leq D:= \sqrt{A^2+B^2}$.
We deduce that for every $x\in U$,
\[
|Q_n(x)| \leq D^N \left(\frac{N}{A} \right)^N \frac{1}{n!(N-n)!},
\]
\[
\begin{aligned}
|P(x)| &\leq \|P\|_I\cdot \sum_{n=0}^N D^N \left(\frac{N }{A} \right)^N \frac{1}{n!(N-n)!}\\
&\leq \|P\|_I \cdot \frac{D^N }{N!} \frac{N^N}{A^N} 2^N\\
&\leq \|P\|_I \cdot (2e\sqrt{1+(B/A)^2})^N,
\end{aligned}
\]
where we used the inequality $N!\geq (N/e)^N$, which is valid for every $N>0$ as a consequence of Robbins formula \cite{robbins_formula}.
\end{proof}

As an easy consequence of the integration trick's lemma, we can state the following:

\begin{cor}
\label{cor_extension}
In the same context as Lemma \ref{lem_integration_trick}, consider a sequence $(h_n)$ of functions $h_n\in \mc{O}((V\setminus \mathbb{R}^+)\cup U)$ satisfying uniform estimates 
\[
\begin{aligned}
&|h_n(p)| \leq C_W\quad \forall p\in D\setminus W,\\
&|h_n(p)| \leq C_U \frac{1}{|\mr{Im}(p)|^N}\quad \forall p\in U\setminus \mathbb{R}^+.
\end{aligned}
\]
Suppose there exists a function $h\in \mc{O}(V\setminus \mathbb{R}^+)$ such that $h_n \rightarrow h$ uniformly on every compact subset of $D\setminus \mathbb{R}^+$.
Then $h$ can be analytically extended to $U$, and 
\[
|h(p)| \leq b_{K,N} \mr{max}(C_W,C_U)\quad \forall p\in U.
\]
\end{cor}

\subsection{Continuity of partial sums}
\label{sec_continuity_partial_sums}

Consider for each open subset $I\subset \mathbb{R}$ and each neighborhood $H\subset \ms{N}(-\infty)$ the subset $E_0^{(I)}(H)$ of bounded holomorphic functions $g$ on $H$, such that the support of the Laplace transform $\mc{L}g$ does not intersect $I$.
If $H$ is a logarithmic neighborhood, this subset is closed in $E_0(H)$ by Corollary \ref{cor_extension}.

If $\beta_1, \beta_2$ belong to $I$, we can consider a simple loop $\gamma$ in $(\mathbb{C}\setminus \mathbb{R}^+)\cup I$ circling once around $[\beta_1,\beta_2]$ in the direct sense, and intersecting $\mathbb{R}$ exactly at $\beta_1$ and $\beta_2$.
We can thus define the partial sum of $g$ as
\[
S_{[\beta_1,\beta_2]}g(w) = \frac{1}{2i\pi}\int_{\gamma} \mc{L}g(p)\mr{e}^{pw}dp.
\]

Taking $\beta_1$ negative gives usual partial sums.

We can define similarly remainders $S_{[\beta_3,\infty[}g$, integrating along the path obtained as the border of $C_+(0,\theta)\cap \{\mr{Re}(p)\geq \beta_3\}$ for some $\theta$.

\begin{thm}
\label{thm_partial_sums}
Consider a logarithmic neighborhood $H\in \ms{LN}(-\infty)$ contained in $\{\mr{Re}(w)\leq 0\}$, two points $\beta_1,\beta_2\in \mathbb{R}$, a number $r>0$ and $I = ]\beta_1-r,\beta_1+r[\cup ]\beta_2-r,\beta_2+r[$.
The application $S_{[\beta_1,\beta_2]}: E_0^{(I)}(H) \rightarrow E_0^{(I)}(-1+H)$ is continuous.

Similarly, if $\beta_3\in \mathbb{R}^+$ and $I=]\beta_3-r,\beta_3+r[$, the application $S_{[\beta_3,+\infty[} : E_0^{(I)}(H) \rightarrow E_0^{(I)}(-1+H)$ is continuous.
\end{thm}

\begin{proof}

Denote by $J=[\beta_1,\beta_2]$ or $J=[\beta_3,+\infty[$ the interval defining $S_J$.
It is enough to treat the case $J = [\beta_3,+\infty[$ with $\beta_3>0$.
Indeed, we can write $S_{[\beta_1,\beta_2]} = S_{[\beta_1,+\infty[}-S_{[\beta_2,+\infty[}$ when $0<\beta_1<\beta_2$, and $S_{[\beta_1,\beta_2]} = id - S_{[\beta_2,+\infty[}$ when $\beta_1<0<\beta_2$.

Consider a point $w_0\in \mathbb{R}^-\cap H$ and put $h=\mc{L}_{w_0}g$.
From the integration trick Lemma \ref{lem_integration_trick}, there is a neighborhood $U$ of $\beta_3$ in $\mathbb{C}$ and a constant $C$ such that $|h(p)|\leq C\|g\|_{H}$ for any $p\in U$ (the constant $C$ does not depend on $g$).
Now, for every $w\in -1+ H$, consider a cone $C_-(w_1,\frac{\pi}{2}-\theta)\subset -1+H$ containing $w$, and the integration path $\gamma$ given by the border of $C_+(0,\theta)\cap \{\mr{Re}(p)\geq \beta_3\}$.
We have 
\[
S_{[\beta_3,+\infty[}g(w) = \frac{1}{2i\pi}\int_{\gamma} h(p)\mr{e}^{pw}dp.
\]
Note that if $w = w_1+w_{12}$, then $\mr{Re}(wp) = w_1\mr{Re}(p)+\mr{Re}(w_{12}p)\leq w_1\mr{Re}(p)$ on the whole path $\gamma$.
We can estimate the integrals along the infinite parts of $\gamma$ as always, and on the vertical part we can consider that $w_1<-1$ to obtain $|h(p)\mr{e}^{wp}|\leq C\|g\|_{H}\mr{e}^{-\mr{Re}(p)}$.
The length of this segment is linear in $\mr{Re}(p)$ so any majoration of $\mr{Re}(p)\mr{e}^{-\mr{Re}(p)}$ allows us to conclude.
\end{proof}

\begin{cor}
\label{cor_limit_series}
Consider subsets $R_n,R\subset \mathbb{R}^+$, enumerated by increasing sequences $(\beta_{n,k})_k, (\beta_k)_k$.
Consider a logarithmic half-plane $H$ and bounded function $g_n,g\in E_0(H)$.

Suppose that $g_n \rightarrow g$ uniformly, that $\beta_{n,k} \underset{n\to\infty}{\longrightarrow} \beta_k$ for each $k$, and that the support of the hyperfunction $\mc{L}g_n$ is contained in $R_n$ for each $n$.
Then the support of $\mc{L}g$ is contained in $R$.

Moreover, for each $\beta_1,\beta_2\notin R$, the partial sums $S_{[\beta_1,\beta_2]}g_n$ are well-defined for $n$ big enough, and $S_{[\beta_1,\beta_2]}g_n \rightarrow S_{[\beta_1,\beta_2]}g$ uniformly on $-1+H$.
\end{cor}

\section{Summation formulas on logarithmic half-planes}
\label{sec_summation_formulas}

As explained in the introduction, the goal of this section is to give summation formulas : when $g(w)$ is a discrete sum of exponentials $g(w) = \sum_{\beta\in R} a_{\beta} \mr{e}^{\beta w}$, how can we compute $g$ from its coefficients $a_{\beta}$ ?

\subsection{The Laplace transform along a logarithmic curve}

We begin by an alternate formula for the Laplace transform, which will conveniently give estimates for its primitives.
We will use a "logarithmic curve" $\ms{C}\subset H$:
\[
\ms{C} = \{ x+iy \in \mathbb{C} \:|\: x= \rho(|y|) \},
\]
for $\rho(y) = -a - k\mr{log}(1+y)$.
We can see it as the border of a logarithmic half-plane (our definition of logarithmic half-planes was slightly different, but is equivalent asymptotically).

We will write $\ms{C}^+ = \ms{C}\cap \{\mr{Im}(w) \geq 0\}$ and $\ms{C}^-=\ms{C}\cap \{\mr{Im}(w)\leq 0\}$.

\begin{lem}
\label{lem_laplace_courbe_log}
Consider a curve $\ms{C}\subset H$ as above and the point of intersection $w_0$ between $\ms{C}$ and $\mathbb{R}$.
For every $p\in \mathbb{C}\setminus \mathbb{R}^+$ with $\mr{Im}(p)\leq 0$ (resp. $\mr{Im}(p)\geq 0$), we have 
\[
\mc{L}_{w_0}g(p) = \int_{\ms{C}^+}g(w)\mr{e}^{-pw}dw\quad (\text{resp. } =\int_{\ms{C}^-}g(w)\mr{e}^{-pw}dw).
\]
\end{lem}

\begin{proof}
Consider a point $p\in \mathbb{C}\setminus \mathbb{R}^+$ with $\mr{Im}(p)\leq 0$ and a vector $v$ with $\mr{Im}(v)>0$ and $\mr{Re}(pv)>0$.
For every point $w_1\in \ms{C}^+$ far enough from $w_0$, the half-line $w_1+\mathbb{R}^+v$ does not cross $\ms{C}$ away from $w_1$, and we can compute $\mc{L}_{w_0}g(p)$ on $\Gamma_{w_0,w_1,v}$.
For convenience, denote by $\ms{C}^+(w_0,w_1)$ the segment of curve on $\ms{C}^+$ between $w_0$ and $w_1$, and by $\ms{C}^+(w_1,\infty)$ the semi-infinite piece of curve on $\ms{C}^+$ starting from $w_1$.
The curve $\Gamma_{w_0,w_1,v}$ can be chosen so that the curve joining $w_0$ and $w_1$ is exactly $\ms{C}^+(w_0,w_1)$, which gives the following expression for $\mc{L}_{w_0}g(p)$:
\[
\begin{aligned}
\mc{L}_{w_0}g(p) &= \int_{w\in\ms{C}^+(w_0,w_1)} g(w)\mr{e}^{-pw}dw + \int_{w\in w_1+\mathbb{R}^+v} g(w)\mr{e}^{-pw}dw\\
&=: I_1(w_1) + I_2(w_1).
\end{aligned}
\]
We want to show that the integral $I_2(w_1)$ tends to zero when $w_1$ tends to infinity on $\ms{C}^+$, and that $I_1(w_1)$ tends to $\int_{\ms{C}^+}g(w)\mr{e}^{-pw}dw$, which amounts to say that
\[
I_3(w_1) := \int_{w\in\ms{C}^+(w_1,\infty)}g(w)\mr{e}^{-pw}dw
\]
tends to zero.

We estimate the integral $I_2$ as always 
\[
\begin{aligned}
|I_2| &\leq \int_{t=0}^{\infty} \|g\|_H \mr{e}^{-\mr{Re}(pw_1+ptv)}|v|dt\\
&\leq \|g\|_{H} \frac{\mr{e}^{-\mr{Re}(pw_1)}|v|}{\mr{Re}(pv)}.
\end{aligned}
\]
Writing $p=p_1+ip_2$ and $w_1 = x + i y$ with $p_2,x\leq 0$ and $y>0$, we get $\mr{Re}(pw_1) = p_1x - p_2 y$.
Since $x \sim -k\mr{log}(1+ y)$ for $w_1\in \ms{C}^+$, we already get $\mr{Re}(pw_1) \rightarrow +\infty$ when $p_2\neq 0$.
If $p_2=0$, then $p\in \mathbb{R}^{-*}$ and $p_1x \rightarrow +\infty$ when $w\to \infty$ on $\ms{C}^+$.
In any case, we got $I_2(w_1) \rightarrow 0$ when $w_1\to \infty$.

To estimate the integral $I_3$, we need to write it explicitely:
\[
I_3(w_1) = \int_{y=y_1}^{\infty} g(w)\mr{e}^{-p(-a-k\mr{log}(1+y))-ip y} \left( -\frac{k d y}{1+y} +id y \right).
\]
For $y_1$ big enough, a straightforward estimation gives:
\[
|I_3(w_1)| \leq \|g\|_\infty \int_{y = y_1}^\infty \mr{e}^{ap_1}(1+y)^{kp_1} \mr{e}^{p_2 y} (k+1) d y;
\]
since $p_2<0$, this integral is finite and tends to zero when $y_1$ tends to infinity, ending the proof of the lemma.
\end{proof}

\begin{lem}
\label{lem_IPP_n}
Suppose that $H\subset \{\mr{Re}(w)\leq -1\}$, and $\ms{C}=\{x+iy\in \mathbb{C} \:|\: x=-a-k\mr{log}(1+|y|)\}\subset H$.
Put $w_0= \ms{C}\cap \mathbb{R}$.
For any bounded holomorphic function $g$ on $H$, any $p\in \mathbb{C}$ with $\mr{Re}(p)>0$ and any $n\in \mathbb{N}$ with $n\geq k\mr{Re}(p)$, we have the estimate 
\[
| I^n\mc{L}_{w_0}g(p) | \leq \frac{(k+1)\mr{e}^{an/k}2^{n}\|g\|_{\infty}}{|\mr{Im}(p)|}.
\]
Moreover, the function $I^n\mc{L}_{w_0}g$ can be continuously extended to the real interval $[0, (n-1)/k[$ from above and from below, giving two continuous functions $(I^n\mc{L}_{w_0}g)^+(\beta)$, $(I^n\mc{L}_{w_0}g)^-(\beta)$.
For $k\mr{Re}(p)\leq n-2$, the following bound is also valid for the extensions to $[0,(n-2)/k]$:
\[
|(I^n\mc{L}_{w_0}g)^\pm(p)| \leq (k+1)\mr{e}^{a \frac{n-2}{k}}2^{\frac{n-2}{k}+1}\|g\|_{\infty}.
\]
The difference between these two extensions is given by the equation 
\[
(I^n\mc{L}_{w_0}g)^+(\beta) - (I^n\mc{L}_{w_0}g)^-(\beta) = \int_{w\in\ms{C}} \frac{g(w)}{w^n}\mr{e}^{-\beta w} dw.
\]
\end{lem}

\begin{proof}
Write $p=p_1 + ip_2$.
From Lemma \ref{lem_ILg} we have
\[
I^n\mc{L}_{w_0}g = (-1)^n \mc{L}_{w_0} \left( \frac{g(w)}{w^n} \right),
\]
note in particular that for $w\in H$, we have $|w|\geq 1$.
Suppose that $p_2<0$ (the case $p_2>0$ can be treated similarly), we can use Lemma \ref{lem_laplace_courbe_log} :
\[
\begin{aligned}
|I^n\mc{L}_{w_0}g(p)| &= \left| \int_{\ms{C}^+} \frac{g(w)}{w^n}\mr{e}^{-pw}dw \right|\\
&\leq \int_{t=0}^{\infty} \|g\|_{\infty} \mr{e}^{ap_1}\frac{(1+t)^{kp_1}}{|\rho(t)+it|^n}\mr{e}^{p_2t} \left| \frac{k}{1+t} + i \right| dt\\
&\leq \int_{t=0}^{\infty} (k+1)\mr{e}^{ap_1}\|g\|_{\infty} \frac{(1+t)^{kp_1}}{(1+t^2)^{\frac{n}{2}}}\mr{e}^{p_2t}dt\\
&\leq (k+1)\mr{e}^{ap_1}\frac{\|g\|_{\infty}2^{kp_1}}{|p_2|}.
\end{aligned}
\]
We used the facts that $(1+t)^2\leq 2(1+t^2)$ and $n\geq kp_1$ to conclude.
This gives the sought estimate for $kp_1\leq n$.
When $n-kp_1>1$, the function 
\[
t\mapsto \frac{(1+t)^{kp_1}}{(1+t^2)^{\frac{n}{2}}}
\]
is integrable, and the limit of $I^n\mc{L}g(p)$ when $p_2$ tends to zero is given by the dominated convergence theorem:
\[
I^n\mc{L}g(p) \underset{p_2\to 0}{\longrightarrow} \int_{w\in \ms{C}^+} \frac{g(w)}{w^n} \mr{e}^{-p_1w}dw.
\]
Denote by $(I^n\mc{L}g)^-$ the limit, as in the statement of this lemma; we see that this function is continuous on $[0,(n-1)/k[$.
We also get the following estimate for $p_1\leq (n-2)/k$:
\[
|I^n\mc{L}g(p)| \leq (k+1)\mr{e}^{ap_1}2^{kp_1}\|g\|_{\infty} \int_{\mathbb{R}^+}\frac{1}{1+t^2}dt.
\]

The same proof for $p_2>0$ gives the expression 
\[
(I^n\mc{L}g)^+(\beta) = \int_{w\in \ms{C}^-} \frac{g(w)}{w^n}\mr{e}^{-\beta w} dw,
\]
so that 
\[
(I^n\mc{L}g)^+(\beta) - (I^n\mc{L}g)^-(\beta) = \int_{w\in \ms{C}} \frac{g(w)}{w^n}\mr{e}^{-\beta w} dw.
\]
\end{proof}

\subsection{Diagonal integration by parts}
\label{sec_DIPP}

Fix a logarithmic neighborhood $H\in \ms{LN}_{a,k}(-\infty)$, a point $w_0$ to define the Laplace transform (we will omit the mention to it in this section to alleviate notations).

Consider the paths $\gamma_{v}=R(-\varepsilon,v)$ and $\gamma_{\bar{v}}=R(-\varepsilon,\bar{v})$ defining the inverse Laplace transform (the parameter $\varepsilon>0$ will be irrelevant, for convenience we can consider that $-\varepsilon = -0.5/k$).
Introduce for $n\geq 1$ the part $\gamma_{v,n}$ of the path $\gamma_{v}$ inside the strip $\{\frac{n-0.5}{k}\leq \mr{Re}(p) \leq \frac{n+0.5}{k}\}$, and the part $\gamma_{v,0}$ inside the strip $\{-\varepsilon \leq \mr{Re}(p)\leq \frac{1}{2k}\}$.
Introduce also the border points $p_n, p_{n+1}$ of the path $\gamma_{v,n}$ (we will omit the dependance of $p_n$ on $v$ at first, but write them $p_{v,n}$ when this dependance becomes relevant).
The border points for the path $\gamma_{\bar{v},n}$ are just $\overline{p_n}$.
Doing $n+3$ integration by parts, we get for $n\geq 0$:

\begin{equation}
\label{eq_IPP_n}
\begin{aligned}
\int_{p\in \gamma_{v,n}} \mc{L}g(p)\mr{e}^{wp}dp &= (I\mc{L}g)(p_{n+1})\mr{e}^{wp_{n+1}} - (I\mc{L}g)(p_n)\mr{e}^{wp_n} - \int_{p\in \gamma_{v,n}} (I\mc{L}g)(p)w\mr{e}^{wp}dp\\
&= \ldots \\
&= \sum_{r=0}^{n+2} (-1)^r \left[ (I^{r+1}\mc{L}g)(p_{n+1})w^r\mr{e}^{wp_{n+1}} - (I^{r+1}\mc{L}g)(p_n)w^r\mr{e}^{wp_n} \right]\\
&\quad + (-1)^{n+3}\int_{p\in \gamma_{v,n}} (I^{n+3}\mc{L}g)(p)w^{n+3}\mr{e}^{wp}dp.
\end{aligned}
\end{equation}

Let us denote the border terms 
\[
BT_n^+(w) := \sum_{r=0}^{n+2}(-1)^r (I^{r+1}\mc{L}g)(p_{n+1})w^r\mr{e}^{wp_{n+1}}
\]
\[
BT_n^-(w) := \sum_{r=0}^{n+2}(-1)^r (I^{r+1}\mc{L}g)(p_{n})w^r\mr{e}^{wp_n},
\]
so that 
\[
\int_{p\in \gamma_{v,n}} \mc{L}g(p)\mr{e}^{wp}dp = BT_n^+(w)-BT_n^-(w) + (-1)^{n+3} \int_{p\in \gamma_{v,n}} (I^{n+3}\mc{L}g)(p)w^{n+3}\mr{e}^{wp}dp.
\]

Whenever the integral $\int_{p\in \gamma_{v}} \mc{L}g(p)\mr{e}^{wp}dp$ converges, we have 
\[
\int_{p\in \gamma_{v}} \mc{L}g(p)\mr{e}^{wp}dp = \sum_{n=0}^\infty \int_{p\in \gamma_{v,n}} \mc{L}g(p)\mr{e}^{wp}dp.
\]
Note that in the sum $\int_{\gamma_{v,n}} \mc{L}g(p)\mr{e}^{wp}dp + \int_{\gamma_{v,n+1}} \mc{L}g(p)\mr{e}^{wp}dp$ we have some cancellations at the point $p_{n+1}$:
\[
\begin{aligned}
BT_n^+(w) - BT_{n+1}^-(w) &= \sum_{r=0}^{n+2} (-1)^r (I^{r+1}\mc{L}g)(p_{n+1}) w^r\mr{e}^{wp_{n+1}}\\
 &\quad - \sum_{r=0}^{n+3} (-1)^{r} (I^{r+1}\mc{L}g)(p_{n+1})w^r\mr{e}^{wp_{n+1}}\\
&= (-1)^{n+4} (I^{n+4}\mc{L}g)(p_{n+1})w^{n+3}\mr{e}^{wp_{n+1}}.
\end{aligned}
\]
We will denote by $BT_{n+1}^0(w) = BT_n^+(w)-BT_{n+1}^-(w)$ these pre-cancelled border terms.

Following these computations, we introduce the diagonal extraction of integration by parts, or "diagonal integration by parts" for short
\[
\begin{aligned}
I_{v}^{\Delta'}(\mc{L}g,\mr{e}^{wp}) &:= \sum_{n=0}^{\infty} (-1)^{n+3}\int_{p\in \gamma_{v,n}} (I^{n+3}\mc{L}g)(p)w^{n+3}\mr{e}^{wp}dp \\
& \quad\quad- BT_0^-(w) 
+ \sum_{n=1}^{\infty} BT_{n}^0(w),
\end{aligned}
\]
and
\[
I_{v}^{\Delta}(\mc{L}g,\mr{e}^{wp}) = I_{\bar{v}}^{\Delta'}(\mc{L}g,\mr{e}^{wp}) -  I_{v}^{\Delta'}(\mc{L}g,\mr{e}^{wp}).
\]
Note already that the border terms $BT_0^-(w)$ cancel themselves in the difference.

\begin{lem}
\label{lem_IPP_CV}
For every angle $\theta_0$ and every $\theta < \frac{\pi}{2} - \theta_0$, there is a cone $C=C_-(w_0,\theta)$ such that the series defining the diagonal integration by parts $I_{v}^{\Delta}(\mc{L}g,\mr{e}^{wp})$ converges normally for every $w\in C$ and any vector $v$ with $\mr{Re}(v)>0$ and $0<|\mr{Arg}(v)|<\theta_0$.
\end{lem}

\begin{proof}
The integral along $\gamma_{v,0}$ can easily be bounded independently; note that in every other term, the only points $p$ considered satisfy $\mr{Re}(p)\geq k/2$, and appear in an exponential or in $I^n\mc{L}g(p)$ for $n\geq k\mr{Re}(p) + 2$.
We can use the estimate of Lemma \ref{lem_IPP_n}:
\[
\begin{aligned}
|I^n\mc{L}g(p)| &\leq 2(k+1)\mr{e}^{(a+\mr{log}(2))\frac{n-2}{k}}\|g\|_{H}\\
&\leq C \mr{e}^{\tilde{a}\frac{n-2}{k}}\|g\|_H.
\end{aligned}
\]
If we denote by $t = \mr{tan} \left( \mr{Arg}(v) \right)$, we obtain the following bound inside the integral on $\gamma_{v,n}$:
\[
\begin{aligned}
|(I^{n+3}\mc{L}g)(p)w^{n+3}\mr{e}^{wp}| &\leq C\|g\|_H\mr{e}^{\tilde{a}\frac{n+1}{k}}|w|^{n+3}\mr{e}^{\mr{Re}(wp)}\\
&\leq C\|g\|_H \mr{e}^{\tilde{a}\frac{n+1}{k}}|w|^{n+3}\mr{e}^{-0.5\mr{Re}(w)/k}\mr{e}^{(p_1+0.5/k)(\mr{Re}(w)-t\mr{Im}(w))}\\
&\leq C\|g\|_H\mr{e}^{\tilde{a}\frac{n+1}{k}}|w|^{n+3}\mr{e}^{-0.5\mr{Re}(w)/k} \mr{e}^{\frac{n}{k} (\mr{Re}(w)-t \mr{Im}(w))}\\
&\leq C\|g\|_H\mr{e}^{\tilde{a}/k}|w|^3\mr{e}^{-0.5\mr{Re}(w)/k} \\
&\quad\quad \times \mr{exp} \left[ \frac{n}{k} \left( \tilde{a} + k\:\mr{log}(|w|) + \mr{Re}(w)-t\mr{Im}(w) \right) \right].
\end{aligned}
\]
(We used $\mr{Re}(w)-t\mr{Im}(w) < 0$ in these inequalities.)
We obtain the following necessary condition on $w$ for the convergence of the series\::
\begin{equation}
\label{eq_IPP_CV}
\mr{Re}(w)-t\mr{Im}(w) + k\:\mr{log}(|w|) + \tilde{a} <0.
\end{equation}
Note that in the cone $C=C_-(w_0,\mr{Arctan}(c))$, we have for every $w=w_0+w_{01}\in C$ 
\[
\mr{Re}(w)-t\mr{Im}(w) \leq w_0 +\mr{Re}(w_{01}) + |t||\mr{Im}\:w_{01}| \leq w_0 + (1-c|t|)\mr{Re}(w_{01}),
\]
and
\[
|w_{01}| \leq \sqrt{1+c^2} |\mr{Re}\:w_{01}|.
\]
Introducing $\displaystyle \tilde{c} =  \frac{1-c|t|}{\sqrt{1+c^2}}$ and $\displaystyle \hat{c} = \frac{1}{\sqrt{1+c^2}}$, we get in the case $\tilde{c}>0$
\[
\begin{aligned}
\mr{Re}(w)-t\:\mr{Im}(w)+k\mr{log}|w| &\leq w_0 - \frac{1-c|t|}{\sqrt{1+c^2}}|w_{01}| + k\:\mr{log}|w|\\
&\leq (1-\tilde{c})w_0 - \tilde{c}(|w_0|+|w_{01}|) + k\:\mr{log}|w|\\
&\leq (1-\hat{c})w_0 +\mr{sup}_{|w|>1} (k\:\mr{log}|w|-\tilde{c}|w|).
\end{aligned}
\]
So for every $t>0$ we can choose any $c< t^{-1}$; denoting by
\[
s = \mr{sup}_{|w|>1} (k\:\mr{log}|w|-\tilde{c}|w|) = k\:\mr{log}(k)-k - k\:\mr{log}(\tilde{c}),
\]
we have the sought result for any $w_0< \frac{-\tilde{a}-s}{1-\hat{c}}$.
Note that the function $\tilde{c}(t)$ increases as $t$ decreases (with $c$ constant), so the function $s(t)$ decreases and the equation \eqref{eq_IPP_CV} is satisfied in the cone $C$ uniformly for every $t'<t$.
\end{proof}

\begin{lem}
For every direction $v$, we have
\[
\int_{\gamma_{v}} \mc{L}g(p)\mr{e}^{wp}dp = I_{v}^{\Delta}(\mc{L}g,\mr{e}^{wp})
\]
whenever both are defined.
\end{lem}

\begin{proof}
In view of the formula for the integration by parts \eqref{eq_IPP_n}, and of the convergence of $I_v^{\Delta}(\mc{L}g,\mr{e}^{wp})$ (Lemma \ref{lem_IPP_CV}), we only need to prove the convergence of the series of boundary terms $\sum |BT_n^+(w)|$ and $\sum |BT_n^-(w)|$.
We will estimate the first sum along the path $\gamma_{v}$, the other terms can be estimated similarly.
For convenience, we consider $v=\mr{e}^{i \theta}/\mr{cos}(\theta)$ so that $p_n=(nv-0.5)/k$.

We can write $I^n\mc{L}g =\mc{L}( \frac{g(w)}{w^n})$ and use $|w|\geq 1$ to estimate (without any regard for subtelty)
\[
|I^n\mc{L}g(p)| \leq \|g\|_H \frac{\mr{e}^{\mu \mr{Re}(p)}}{N(p)}
\] 
for any $p$ on the path $\gamma_{v}$, and some constant $\mu>0$.
A bit farther on the path of non-subtelty, we find the estimate $\mr{Im}(p_n)\geq k\:\mr{tan}(\theta)/2$, which gives 
\[
\begin{aligned}
\sum_{n\geq 1} |BT_n^+(w)| &= \sum_{n\geq 1} \sum_{r=0}^{n+2} |(I^{r+1}\mc{L}g)(p_{n+1})| |w|^r |\mr{e}^{wp_{n+1}}|\\
&\leq \sum_{n\geq 1}\sum_{r=0}^{n+2} \frac{2\|g\|_H}{k\:\mr{tan}(\theta)} |w|^r\mr{e}^{\mr{Re}(wv)(n+1)/k-0.5\mr{Re}(w)/k}\mr{e}^{\mu (n+0.5)/k}\\
&\leq \frac{2\|g\|_H}{k\:\mr{tan}(\theta)} \mr{e}^{\mu/(2k)}\mr{e}^{-\frac{\mr{Re}(w)}{2k}} \sum_{n\geq 1} \frac{|w|^{n+3}-1}{|w|-1}\mr{e}^{\frac{n}{k} \left( \mr{Re}(wv)-\mu \right)}.
\end{aligned}
\]
Once again, we find that this series converges under the hypothesis 
\[
\mr{Re}(wv)-\mu + k\:\mr{log}(w) < 0;
\]
we deduce as in Lemma \ref{lem_IPP_CV} that this sum converges for $w$ in some cone.

We conclude by analytic continuation.
\end{proof}

In particular, the diagonal integration by parts $I_v^{\Delta}(\mc{L}g,\mr{e}^{wp})$ does not depend on the vector $v$.
But the formula defining it admits a limit when $v$ tends to the real point $1$.

Introduce the continuous function
\[
\varphi_n(t) = (I^n\mc{L}g)^-(t)-(I^n\mc{L}g)^+(t)
\]
defined for $t\in [0,(n-2)/k]$ by Lemma \ref{lem_IPP_n}.
Introduce the limit of the diagonal integration by parts 
\[
I_1^{\Delta}(\mc{L}g,\mr{e}^{wt}) := \sum_{n=0}^{\infty} (-1)^{n+3}\int_{t=t_n}^{t_{n+1}} \varphi_{n+3}(t)w^{n+3}\mr{e}^{wt}dt + \sum_{n=1}^\infty BT_n^1(w),
\]
where the cutting points are $t_n = \frac{n-0.5}{k}$ and the boundary terms are 
\[
BT_n^1(w) = (-1)^{n+3} \varphi_{n+3}(t_n)w^{n+2}\mr{e}^{wt_n}.
\]

\begin{thm}
\label{thm_DIPP}
The series defining $I_1^{\Delta}(\mc{L}g,\mr{e}^{wt})$ converges normally in a logarithmic neighborhood $H'$, and for every $w\in H'$ we have 
\[
\mc{L}^{-1}\mc{L}g(w) = \frac{1}{2i\pi}I_1^{\Delta}(\mc{L}g,\mr{e}^{wt}).
\]
\end{thm}

\begin{proof}
As explained before, we only have to let $v$ tend to $1$ in Lemma \ref{lem_IPP_CV}.
More explicitely, by Lemma \ref{lem_IPP_n}, we have for every $t\leq (n-2)/k$ 
\[
\begin{aligned}
|\varphi_n(t)| &\leq 2(k+1)\mr{e}^{(a+\mr{log}(2)) \frac{n-2}{k}}\|g\|_H\\
&\leq C\mr{e}^{\tilde{a}n/k}\|g\|_H.
\end{aligned}
\]
We can then estimate 
\[
\begin{aligned}
|\varphi_{n+3}(t)w^{n+3}\mr{e}^{wt}| &\leq C \|g\|_H\mr{e}^{\tilde{a}n/k}|w|^{n+3}\mr{e}^{t\:\mr{Re}(w)}\\
&\leq C\|g\|_H\mr{e}^{\tilde{a}n/k}|w|^{n+3}\mr{e}^{\frac{n-2}{k}\mr{Re}(w)}\\
&\leq C\|g\|_H|w|^3\mr{e}^{-2\mr{Re}(w)/k} \mr{exp} \left[ \frac{n}{k} ( \tilde{a} + k\:\mr{log}|w| + \mr{Re}(w)) \right].
\end{aligned}
\]
The series converges normally as soon as 
\[
\tilde{a} + k\:\mr{log}|w| + \mr{Re}(w) <0.
\]
As explained in the Appendix \ref{appendix_logarithmic}, this is a logarithmic neighborhood.
\end{proof}

Since the integral $I_{1}^{\Delta}(\mc{L}g,\mr{e}^{wt})$ depends continuously on the function $g$, we obtain an explicit method for computing $g$ from its formal development $\hat{g}$ : if $\hat{g}(w) = \sum_{\beta\in R} a_{\beta} \mr{e}^{\beta w}$ for some discrete subset $R\subset \mathbb{R}^+$, we see that
\[
\varphi_n(t) = \sum_{\beta\leq t} a_{\beta} \frac{(t-\beta)^{n-1}}{(n-1)!}
\]

The diagonal integration by parts $I_1^{\Delta}$ gives an explicit summation formula in terms of the formal development $\hat{g}(w)$.
In general, we say that a formal series $\hat{g}$ can be re-summated by a diagonal integration by parts when $I_1^{\Delta}$ converges for $w$ in a neighborhood of $-\infty$.

\begin{cor}
If a bounded function $g\in \mathcal{O}(H)$ is a uniform limit of normally convergent series $g_n=\sum_{\beta\in R_{n}} a_{\beta,n}\mr{e}^{\beta w}$, and the formal development of $g$ at $-\infty$, $\hat{g}(w) = \sum_{\beta\in R} a_{\beta}\mr{e}^{\beta w}$ has a closed discrete set of indices $R$, then the formal series $\hat{g}$ can be re-summated by a diagonal integration by parts.
Moreover, $I_1^{\Delta}(\mc{L}\hat{g},\mr{e}^{wt}) = g(w)$ whenever both are defined.
\end{cor}

\subsection{Evanescent summation}

In this section, we explicit in which sense a function $g\in E_0(H)$ is the limit of its partial sums.
We will use notations from sections \ref{sec_continuity_partial_sums} and \ref{sec_DIPP}.
Using the points $t_n$ of section \ref{sec_DIPP} as cutting points for the partial sums, we obtain some bounded functions $S_{[t_0,t_n]}g\in E_0(-1+H)$.
We will also suppose here that each point $t_n$ is not in the support of $g$ for partial sums to behave correctly.
To simplify notations, we will write $S_ng = S_{[t_0,t_n]}g$, and $H'=-1+H$.

We can write $S_ng$ as the integral along a closed path $\gamma$ around the segment $[-1,t_n]$ and cutting $\gamma$ at the point $t_n$ to make $n+3$ integration by parts:
\[
\begin{aligned}
\int_{\gamma} \mc{L}g(p)\mr{e}^{wp}dp &= (-1)^{n+3} \int_{\gamma\setminus\{t_n\}} (I^{n+3}\mc{L}g)(p)w^{n+3}\mr{e}^{wp}dp \\
&\quad - \sum_{r=0}^{n+2} (-1)^r(I^{r+1}\mc{L}g)^+(t_n)w^r\mr{e}^{wt_n} + \sum_{r=0}^{n+2} (-1)^r(I^{r+1}\mc{L}g)^-(t_n)w^r\mr{e}^{wt_n}.
\end{aligned}
\]
Once again, $I^{n+3}\mc{L}g$ extends continuously to the segment $[t_0,t_n]$ from above and from below, and when $\gamma$ shrinks toward this segment, we obtain the formula 
\[
2i\pi S_ng(w) = (-1)^{n+3} \int_{t=t_0}^{t_n} \varphi_{n+3}(t) w^{n+3}\mr{e}^{wt}dt + BT_n^2(w),
\]
where the border term is 
\[
BT_n^2(w) = \sum_{r=0}^{n+2} (-1)^r\varphi_{n+3}^{(n-r+2)}(t_n)w^r\mr{e}^{wt_n}.
\]
Note that $\mc{L}g$ is holomorphic in a neighborhood of $t_n$ so that $\varphi_{n+3}$ is polynomial of order $n+2$ around $t_n$; the term $\varphi_{n+3}^{(n-r+2)}$ is the derivative of this polynomial.

\begin{df}
We define the $n$-th evanescent partial sum of $g$ as
\[
\tilde{S}_ng(w) = S_ng(w) - \frac{1}{2i\pi} BT_n^2(w).
\]
\end{df}

These evanescent partial sums are in fact the partial sums corresponing to the diagonal integration by parts:

\begin{lem}
For every $n\geq 1$, we have 
\[
2i\pi \tilde{S}_{n+1}g(w) = \sum_{r=0}^{n}(-1)^{r+3} \int_{t=t_r}^{t_{r+1}} \varphi_{r+3}(t)w^{r+3}\mr{e}^{wt}dt + \sum_{r=1}^{n} BT_r^1(w).
\]
\end{lem}

This lemma is an easy application of the classical integration by parts.
From this we obtain immediately:

\begin{thm}
\label{thm_evanescent}
The sequence of evanescent partial sums $\tilde{S}_ng$ tends uniformly to $g$ on a logarithmic neighborhood of $-\infty$. 
\end{thm}

Note that if the function $g$ has formal development $\hat{g}(w) = \sum_{\beta\in R} a_{\beta}\mr{e}^{\beta w}$ for some discrete subset $R\subset \mathbb{R}^+$, then 
\[
\tilde{S}_ng(w) = \sum_{\beta\leq t_n} a_{\beta} \mr{e}^{\beta w} + \sum_{r=0}^{n+2} b_r w^r\mr{e}^{wt_n}
\]
for the constants $b_r=(-1)^{r+1}\varphi_{n+3}^{(n-r+2)}(t_n) \in \mathbb{C}$, in particular the formal development of $\tilde{S}_ng$ coincides with that of $g$ up to order $t_n$.
In fact the difference between $\tilde{S}_ng$ and $S_ng$ is supported on $t=t_n$: the hyperfunction $\mc{L} (BT_n^2)$ is a hyperfunction of order $n+2$ supported on $\{t_n\}$.

\appendix

\section{Equations of logarithmic neighborhoods}
\label{appendix_logarithmic}

Consider the following neighborhoods of $-\infty$:
\[
H_{a,k} = \{ w = x+iy\in \mathbb{C} \:|\: x< a - k\:\mr{log}(1+|y|) \},
\]
\[
\widetilde{H}_{a,k} = \{ w = x+iy \in \mathbb{C} \:|\: x + \frac{k}{2}\:\mr{log}(x^2+y^2) < a \}.
\]

These two families of neighborhoods are equivalent in the sense of the following lemma.
In particular, they can both be used as a definition for logarithmic neighborhoods.

\begin{lem}
If $a\leq \mr{min}(-1,-k)$, then $\widetilde{H}_{a,k}$ is of the form $\{ x < \rho(|y|)\}$ for some function $\rho$.
Moreover, $\widetilde{H}_{a,k}\subset H_{a+1,k}$ and there exists an explicitable function $(a,k) \mapsto M_{a,k}$ such that $H_{a,k}\subset \widetilde{H}_{a+M_{a,k},k}$.
\end{lem}

\begin{proof}
Denote by $F(x,y) = x + k\:\mr{log}(1+|y|)$ and $\tilde{F}(x,y) = x + \frac{k}{2}\mr{log}(x^2+y^2)$.
Note first that 
\[
\frac{\partial_{}\tilde{F}}{\partial_{}x }(x,y) = 1 + \frac{k}{2}\frac{2x}{x^2+y^2} = \frac{x^2+kx+y^2}{x^2+y^2};
\]
It follows that for $x<-k$, if $x+iy\in \widetilde{H}_{a,k}$ the whole interval $]-\infty,x] +iy$ is included in $\widetilde{H}_{a,k}$.

Consider a point $x+iy\in \widetilde{H}_{a,k}$. Since $x<-1$, we have $\mr{log}(x^2+y^2) \geq \mr{log}(1+y^2) \geq 2\:\mr{log}(1+|y|)-\mr{log}(2)$ (from Cauchy-Schwarz's inequality, we have $(1+y)^2\leq 2 (1 + y^2)$).
Thus, 
\[
x + k\:\mr{log}(1+|y|) \leq x + \frac{k}{2}\mr{log}(x^2+y^2) + \frac{\mr{log}(2)}{2},
\]
and the point $x+iy$ belongs to $H_{a',k}$ where $a'=a+\frac{\mr{log}(2)}{2}$.

To study the other inclusion, it suffices to study the restriction of $\tilde{F}$ to the curve $C_a = \partial_{}H_{a,k} = \{ x = a-k\:\mr{log}(1+|y|)\}$.
By symmetry, it suffices to study its restriction to the upper part $\{y>0\}$.
We have for $y>0$ and $x=a-k\:\mr{log}(1+y)$,
\[
\frac{\partial_{}x }{\partial_{}y} = \frac{-k}{1+y},
\]
\[
\begin{aligned}
\frac{\partial_{}\tilde{F}\vert_{C_a}}{\partial_{}y} &= \frac{\partial_{}x}{\partial_{}y} + \frac{k}{2} \frac{2x \frac{\partial_{}x }{\partial_{}y }}{x^2+y^2}+ \frac{k}{2} \frac{2y }{x^2+y^2}\\
&= \frac{-k}{1+y} + \frac{ky}{x^2+y^2} -k^2 \frac{a-k\:\mr{log}(1+y)}{(1+y)(x^2+y^2)}\\
&= k\frac{y-x^2}{(1+y)(x^2+y^2)} -\frac{ak^2}{(1+y)(x^2+y^2)} + k^3 \frac{\mr{log}(1+y)}{1+y}\frac{1}{x^2+y^2}.
\end{aligned}
\]
Each of these three terms is integrable, and so $\tilde{F}\vert_{C_a}$ is bounded.
More precisely, since $x>1$, we have 
\[
\int_0^{\infty} \frac{ak^2}{(1+y)(x^2+y^2)} dy \leq ak^2 \int_0^{\infty} \frac{dy}{1+y^2}\leq \frac{\pi a k^2}{2},
\]
\[
\int_0^{\infty} k^3 \frac{\mr{log}(1+y)}{1+y} \frac{dy}{x^2+y^2}\leq k^3 \int_0^{\infty} \frac{dy}{1+y^2}\leq \frac{\pi k^3}{2};
\]
to estimate the first term, we can use $x^2 = a^2-2ak\:\mr{log}(1+y) + k^2\:\mr{log}(1+y)^2$, and estimate $\frac{1}{1+y}\leq 1$, $\frac{\mr{log}(1+y)}{1+y}\leq 1$ and $\frac{\mr{log}(1+y)^2}{1+y}\leq 1$ (the maximum of this last function is $4/\mr{e}^2$).
Thus
\[
\int_0^{\infty} k \frac{y-x^2}{(1+y)(x^2+y^2)}dy \leq k(1+a^2+2ak+k^2) \frac{\pi}{2}.
\]
Since $\tilde{F}(a,0) = a + k\:\mr{log}(|a|)$, we get
\[
\mr{max}_{C_a}|\tilde{F}-a-k\:\mr{log}(|a|)| \leq k \frac{\pi}{2}(1+a^2+3ak+2k^2)
\]
giving the constant $M_{a,k}$ of the lemma.
\end{proof}

\bibliography{mybib}{}

@incollection{ecalle_small_denominators,
 author = {Ecalle, Jean},
 title = {Compensation of small denominators and ramified linearisation of local objects},
 booktitle = {Complex analytic methods in dynamical systems. Proceedings of the congress held at Instituto de Matem\'atica Pura e Aplicada, IMPA, Rio de Janeiro, Brazil, January 1992},
 pages = {135--199},
 year = {1994},
 publisher = {Paris: Soci{\'e}t{\'e} Math{\'e}matique de France},
 language = {English},
 url = {smf4.emath.fr/Publications/Asterisque/1994/222/html/smf_ast_222_135-199.html},
 zbMATH = {671754},
 Zbl = {0810.58036}
}

@article{ilyashenko_centennial,
 author = {Ilyashenko, Yu.},
 title = {Centennial history of {Hilbert}'s 16th {Problem}},
 fjournal = {Bulletin of the American Mathematical Society. New Series},
 journal = {Bull. Am. Math. Soc., New Ser.},
 issn = {0273-0979},
 volume = {39},
 number = {3},
 pages = {301--354},
 year = {2002},
 language = {English},
 doi = {10.1090/S0273-0979-02-00946-1},
 zbMATH = {1756730},
 Zbl = {1004.34017}
}

@book{morimoto_hyperfunctions,
 author = {Morimoto, Mitsuo},
 title = {An introduction to {Sato}'s hyperfunctions. {Transl}. from the {Japanese} by {Mitsuo} {Morimoto}},
 fseries = {Translations of Mathematical Monographs},
 series = {Transl. Math. Monogr.},
 issn = {0065-9282},
 volume = {129},
 isbn = {0-8218-4571-3},
 year = {1993},
 publisher = {Providence, RI: American Mathematical Society},
 language = {English},
 zbMATH = {482841},
 Zbl = {0811.46034}
}

@article{robbins_formula,
 author = {Robbins, Herbert},
 title = {A remark on {Stirling}'s formula},
 fjournal = {American Mathematical Monthly},
 journal = {Am. Math. Mon.},
 issn = {0002-9890},
 volume = {62},
 pages = {26--29},
 year = {1955},
 language = {English},
 doi = {10.2307/2308012},
 zbMATH = {3115226},
 Zbl = {0068.05404}
}

@article{sato_hyperfunctions1,
 author = {Sato, Mikio},
 title = {Theory of hyperfunctions. {I}},
 fjournal = {Journal of the Faculty of Science. Section I},
 journal = {J. Fac. Sci., Univ. Tokyo, Sect. I},
 issn = {0368-2269},
 volume = {8},
 pages = {139--193},
 year = {1959},
 language = {English},
 zbMATH = {3143096},
 Zbl = {0087.31402}
}

@article{sato_hyperfunctions2,
 author = {Sato, Mikio},
 title = {Theory of hyperfunctions. {II}},
 fjournal = {Journal of the Faculty of Science. Section I},
 journal = {J. Fac. Sci., Univ. Tokyo, Sect. I},
 issn = {0368-2269},
 volume = {8},
 pages = {387--437},
 year = {1960},
 language = {English},
 zbMATH = {3158986},
 Zbl = {0097.31404}
}

@book{yoccoz_diviseurs,
 author = {Yoccoz, Jean-Christophe},
 title = {Petits diviseurs en dimension 1},
 fseries = {Ast{\'e}risque},
 series = {Ast{\'e}risque},
 issn = {0303-1179},
 volume = {231},
 year = {1995},
 publisher = {Paris: Soci{\`e}t{\'e} Math. de France},
 language = {French},
 zbMATH = {827357},
 Zbl = {0836.30001}
}
\bibliographystyle{acm}

\vfill
\textsc{Universidade Federal Fluminense - Instituto de Matemática e Estatística, Niterói, Brasil}

\textit{Email :} olivier\_thom@id.uff.br

\end{document}